\documentclass[reqno]{amsart}

\usepackage{amsfonts}
\usepackage{amssymb}
\usepackage{amsmath}
\usepackage{enumitem}
\usepackage{xcolor}
\usepackage{todonotes}

\setenumerate{label={\rm (\alph{*})}}
\usepackage{bbm,euscript,mathrsfs}
\usepackage{paper_diening}
\usepackage{graphicx}
\usepackage[top=1in, bottom=1.25in, left=1.10in, right=1.10in]{geometry}

\allowdisplaybreaks

\numberwithin{equation}{section}

\usepackage{tikz-cd}
\usetikzlibrary{lindenmayersystems}
\usetikzlibrary{decorations.pathreplacing}

\DeclareMathOperator{\dist}{dis}
\DeclareMathOperator{\Bog}{Bog}

\DeclareMathOperator{\Div}{div}

\newcommand{\R}{\mathbb R}
\newcommand{\N}{\mathbb N}
\newcommand{\dd}{\mathrm d}
\newcommand{\dt}{\,\mathrm{d} t}

\renewcommand{\bfB}{B}

\newcommand{\D}{\mathrm{D}}
\newcommand{\dx}{\,\mathrm{d}x}
\newcommand{\dy}{\,\mathrm{d}y}

\newcommand{\dxt}{\,\mathrm{d}x\,\mathrm{d}t}
\newcommand{\ds}{\,\mathrm{d}\sigma}
\newcommand{\dxs}{\,\mathrm{d}x\,\mathrm{d}\sigma}
\newcommand{\Oeta}{\Omega_\eta}

\newtheorem{theorem}{Theorem}[section]
\newtheorem{lemma}[theorem]{Lemma}
\newtheorem{proposition}[theorem]{Proposition}

\newtheorem{remark}[theorem]{Remark}

\theoremstyle{definition}
\newtheorem{definition}[theorem]{Definition}

\begin{document}

\title[Regularity for 2D FSI]
{Regularity results in 2D fluid-structure interaction}

\author{Dominic Breit}
\address{Department of Mathematics, Heriot-Watt University, Edinburgh, EH14 4AS, United Kingdom}
\email{d.breit@hw.ac.uk}
\address{Institute of Mathematics, TU Clausthal, Erzstra\ss e 1, 38678 Clausthal-Zellerfeld, Germany}
\email{dominic.breit@tu-clausthal.de}



\date{\today}



\begin{abstract}
We study the interaction of an incompressible fluid
in two dimensions with an elastic structure yielding the moving boundary of the physical domain. The displacement of the structure is described by a linear viscoelastic beam equation.
Our main result is the existence of a unique global strong solution.
Previously, only the ideal case of a flat reference geometry was considered such that the structure can only move in vertical direction. We allow for a general geometric set-up, where the structure can even occupy the complete boundary.

Our main tool -- being of independent interest -- is a maximal regularity estimate for the steady Stokes system in domains with minimal boundary regularity. In particular, we can control the velocity in $W^{2,2}$ in terms of a forcing in $L^2$ provided the boundary belongs roughly to $W^{3/2,2}$.
 This is applied to the momentum equation in the moving domain (for a fixed time) with the material derivative as right-hand side. Since the moving boundary belongs a priori only to the class $W^{2,2}$, known results do not apply here as they require a $C^2$-boundary.
 

\end{abstract}

\maketitle

\section{Introduction}

\subsection{The fluid-structure interaction problem}
The interactions of fluids with elastic structures are important for many applications
ranging from hydro- and aero-elasticity over bio-mechanics to hydrodynamics. We are interested in the case, where a viscous incompressible fluid interacts with a flexible shell which is located at one part of the boundary (or even describes the complete boundary) of the underlying domain $\Omega\subset\R^2$ denoted by $\omega$. The shell, described by a function $\eta:(0,T)\times\omega\rightarrow\R$, reacts to the surface forces induced by the fluid
and deforms the domain $\Omega$ to $\Omega_{\eta(t)}$, where the function $\bfvarphi_{\eta(t)}$ describes the coordinate transform (see figures 1 and 2 below) and $\bfn_\eta$ is the normal at the deformed boundary.

The motion of the fluid is governed by the Navier--Stokes equations
 \begin{align} \varrho_f\big(\partial_t \bfu+(\bfu\cdot\nabla)\bfu\big)&=\mu\Delta\bfu-\nabla \pi+\bff,\quad \Div\bfu=0,&\label{1}
 \end{align}
 in the moving domain $\Omega_\eta$, where $\bfu:(0,T)\times\Omega_\eta\rightarrow\R^2$ is the velocity field and $\pi:(0,T)\times\Omega_\eta\rightarrow\R$ the pressure function. The function $\bff:(0,T)\times\Omega_\eta\rightarrow\R^2$ is a given volume force. The equations are
 supplemented with initial conditions
 and the boundary condition $\bfu\circ \bfvarphi_\eta=\partial_t\eta\bfn$ at the flexible part of the boundary with normal $\bfn$.
 There exist various models in literature to model the behaviour of the shell and a typical example is given by
 \begin{align}\label{2}
\varrho_s \partial_t^2\eta-\gamma\partial_t\partial_y^2\eta+\alpha\partial_y^4\eta=g-\bfn \bftau\circ\bfvarphi_\eta\bfn_\eta\,\dd y_{\bfn_\eta}
 \end{align}
on $\omega$ supplemented with initial and boundary conditions. Here $\varrho_s,\gamma$ and $\alpha$ are positive constants and the function $g:(0,T)\times \omega\rightarrow$ is a given forcing term. Here $\bftau$ denotes the Cauchy stress of the fluid given by Newton's rheological law, that is
$\bftau=\mu\big(\nabla\bfu+\nabla\bfu^\top\big)-\pi\mathbb I_{2\times 2}$. The model \eqref{1}--\eqref{2} has been suggested, in particular, for blood vessels (where the 2D geometry is often sufficient), see \cite{Q1,Q2}. \\
There exists already results concerning the existence of local-in-time strong solutions to the coupled system \eqref{1}--\eqref{2}, see \cite{CS,DeSa,GraHilLe,Le,DRR}\footnote{Some of these results are concerned with the 3D case, where global existence of strong solutions is out of reach.}. Rather recently, even the existence of a global-in-time strong solution has been shown in \cite{GraHil}. All these papers are concerned with a simplified geometrical set-up, where the domain $\Omega$ is given by a rectangle and the flexible part of the boundary is flat, see figure 1. In this case the transformation between the reference domain and the moving domain is particularly easy, which simplifies the mathematical analysis significantly. While it is natural to start the investigation which such an idealised model, this model is not suitable for most real-world applications such as blood vessels. In the case of a more realistic non-flat geometry as in figure 2 only the existence of weak solutions to \eqref{1}--\eqref{2} is known, see \cite{LeRu,MuCa1,MuSc}$^1$. For a weak solution, the kinetic energy $\|\bfu\|_{L^2(\Omega_\eta)}$ is bounded, the velocity gradient belongs to $L^2$ and we have
\begin{align}\label{eq:290122}
\eta \in W^{1,\infty} \big(I; L^2(\omega) \big)\cap W^{1,2} \big(I; W^{1,2}(\omega) \big)\cap  L^\infty \big(I; W^{2,2}(\omega) \big).
\end{align}
This can be seen formerly by testing the momentum equation by $\bfu$ and the shell equation by $\partial_t\eta$ (note that the boundary terms cancel due to the condition $\bfu\circ \bfvarphi_\eta=\partial_t\eta\bfn$). 
With a weak solution at our disposal we are confronted with the question whether it enjoys additional regularity properties (in this case we speak about strong solutions) and is, in fact, unique. These properties are not only of theoretical interest but also crucial for robust numerical simulations.
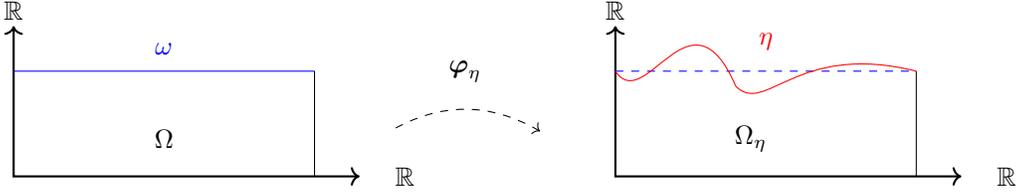
\begin{figure}\label{fig:1}
\begin{center}
\begin{tikzpicture}[scale=2]
  \begin{scope}     \draw [thick, <->] (0.5,0.5) -- (0.5,-0.5) -- (2.8,-0.5);
        \node at (0.5,0.6) {$\R$};
        \node at (1.5,-0.25) {$\Omega$};
        \node [blue] at (1.5,0.35) {$\omega$};
       \node at (3.1,-0.5) {$\R$};
        \draw [blue] (0.5,0.2) -- (2.5,0.2);
                \draw (2.5,0.2) -- (2.5,-0.5);
                   \node at (3.5,0.2) {$\bfvarphi_\eta$};
        \draw [thick, <->] (4.5,0.5) -- (4.5,-0.5) -- (6.8,-0.5);
        \node at (4.5,0.6) {$\R$};
       \node at (7.1,-0.5) {$\R$};
       \draw [blue,dashed] (4.5,0.2) -- (6.5,0.2);
                \draw (6.5,0.2) -- (6.5,-0.5);
                \draw [red] (4.5,0.2) .. controls (4.7,-0.1) and (5,0.8) .. (5.3,0.1);
                \draw [red] (5.3,0.1) .. controls (5.5,-0.1) and  (5.7,0.4) .. (6.5,0.2);
                        \node at (5.4,-0.25) {$\Omega_\eta$};
                         \node [red] at (5.5,0.4) {$\eta$};
          \draw [<-,dashed] (4,-0.2) to [out=150,in=30] (3,-0.2);
  \end{scope}
\end{tikzpicture}
\caption{Domain transformation in the simplified set-up.}
\end{center}
\end{figure}
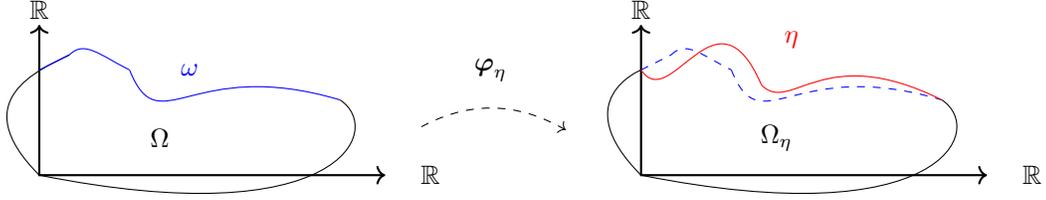
\begin{figure}\label{fig:2}
\begin{center}
\begin{tikzpicture}[scale=2]
  \begin{scope}     \draw [thick, <->] (0.5,0.5) -- (0.5,-0.5) -- (2.8,-0.5);
        \node at (0.5,0.6) {$\R$};
        \node at (1.3,-0.25) {$\Omega$};
          \node [blue] at (1.5,0.2) {$\omega$};
       \node at (3.1,-0.5) {$\R$};
       \draw (0.5,-0.5) .. controls (0,0) and (0.5,0.2) .. (0.7,0.3);
     \draw [blue] (0.5,0.2) .. controls (0.7,0.3) and (0.7,0.3) .. (0.7,0.3);
        \draw [blue] (0.7,0.3) .. controls (0.8,0.4) and (0.9,0.3) .. (1.1,0.2);
       \draw [blue] (1.1,0.2) .. controls (1.3,-0.3) and (1.5,0.3) .. (2.5,0);
       \draw (2.5,0) .. controls (2.8,-0.2) and (2.5,-0.9) .. (0.5,-0.5);
                   \node at (3.5,0.2) {$\bfvarphi_\eta$};
        \draw [thick, <->] (4.5,0.5) -- (4.5,-0.5) -- (6.8,-0.5);
        \node at (4.5,0.6) {$\R$};
       \node at (7.1,-0.5) {$\R$};
                     \draw (4.5,-0.5) .. controls (4,0) and (4.5,0.2) .. (4.5,0.2);
       \draw (6.5,0) .. controls (6.8,-0.2) and (6.5,-0.9) .. (4.5,-0.5);
        \draw [blue,dashed] (4.5,0.2) .. controls (4.7,0.3) and (4.7,0.3) .. (4.7,0.3);
        \draw [blue,dashed] (4.7,0.3) .. controls (4.8,0.4) and (4.9,0.3) .. (5.1,0.2);
       \draw [blue,dashed] (5.1,0.2) .. controls (5.3,-0.3) and (5.5,0.3) .. (6.5,0);
                \draw [red] (4.5,0.2) .. controls (4.7,-0.1) and (5,0.8) .. (5.3,0.1);
                \draw [red] (5.3,0.1) .. controls (5.5,-0.1) and  (5.7,0.4) .. (6.5,0);
                        \node at (5.4,-0.25) {$\Omega_\eta$};
                         \node [red] at (5.5,0.4) {$\eta$};
          \draw [<-,dashed] (4,-0.2) to [out=150,in=30] (3,-0.2);
  \end{scope}
\end{tikzpicture}
\caption{Domain transformation in the general set-up.}
\end{center}
\end{figure}
The analysis of regularity properties of solutions to \eqref{1}--\eqref{2} in the framework of figure 2 is the purpose of the present paper.
Our first main results shows that under natural assumptions on the data there is a unique global-in-time strong solution to \eqref{1}--\eqref{2}, see Theorem \ref{thm:main} for the precise statement. Here strong means that both equation hold in the strong sense, that is, all quantities exist as measurable functions. In particular, all terms in
the momentum equation \eqref{1} belong, in fact, to $L^2$ and we have
\begin{align}\label{eq:www}
\eta \in W^{1,\infty} \big(I; W^{1,2}(\omega) \big)\cap W^{1,2} \big(I; W^{2,2}(\omega) \big)\cap  L^\infty \big(I; W^{3,2}(\omega) \big) \cap W^{2,2}(I;L^2(\omega)).
\end{align}

\subsection{Stokes systems in irregular domains}
As in \cite{GraHil}
the crucial tool in our analysis of \eqref{1}--\eqref{2} is an elliptic estimate for the Stokes system
 \begin{align} \mu\Delta\bfu-\nabla \pi=-\bfg,\quad \Div\bfu=0,&\label{3}
 \end{align}
in a bounded domain $\mathcal O\subset\R^n$ (supplemented with homogeneous boundary conditions). To be more precise, we require an inequality of the form\footnote{This result is independent of the dimension.}
\begin{align}\label{4}
\|\bfu\|_{W^{2,p}(\mathcal O)}+\|\pi\|_{W^{1,p}(\mathcal O)}\lesssim \|\bfg\|_{L^p(\mathcal O)}
\end{align}
for $p\in(1,\infty)$ (in fact, $p=2$ is sufficient for the application to \eqref{1}--\eqref{2}). Such an estimate is well-known if the boundary of the underlying domains belongs to the class $C^2$. 
 We will apply \eqref{4} to \eqref{1} with $\bfg=\varrho_f\big(\partial_t \bfu+(\bfu\cdot\nabla)\bfu\big)$ and $\mathcal O=\Omega_{\eta(t)}$ for a fixed $t$. Hence the regularity of $\mathcal O$ is determined by $\eta$ which only belongs to $W^{2,2}$, see \eqref{eq:290122}. A version of \eqref{4} for the simplified framework from figure 1 is proved in \cite{GraHil}. It is, however, based on some cancellations which are not available in the general case. On the other hand, the question about minimal assumptions
on the regularity of $\partial\mathcal O$ for \eqref{4} is
of independent interest and seems to be missing in literature.
The only comparable result can be found in \cite{BS} (which is, in turn, based on results from \cite{FKV}), where an estimate for Lipschitz domains (that is, $\partial\mathcal O\in W^{1,\infty}$) is shown which controls fractional derivatives (of order 3/2 for $\bfu$ and 1/2 for $\pi$). The method from \cite{BS} is, unfortunately, designed specifically for Lipschitz domains and does not seem to apply in a more general framework.

In Theorem \ref{thm:stokessteady} we offer an exhaustive picture concerning
the maximal regularity theory for the Stokes system \eqref{3}
in irregular domains in the framework of fractional Sobolev spaces. This is based on the theory of Sobolev multipliers from \cite{MaSh} which has not been used in fluid mechanics before.
Our assumptions on the boundary coincide with those made in \cite[Section]{MaSh} for the Laplace equation which are known to be optimal. As a special case we obtain estimate \eqref{4} for $p=2$ provided the Lipschitz constant of $\partial\mathcal O$ is small and $\partial\mathcal O$ belongs -- roughly speaking -- to the class $W^{3/2,2}$ (we will make these concepts precise in Section \ref{sec:para}).
The relation between both spaces is that $W^{3/2,2}$ is the trace-space of $W^{2,2}$ (the space for the velocity field in \eqref{4}) in the sense that the linear mapping
\begin{align}\label{eq:traceembedding}
W^{2,2}(\mathbb R^n)\ni\varphi\mapsto \varphi(\cdot,0)\in W^{3/2,2}(\mathbb R^{n-1})
\end{align}
is continuous. 

\subsection{The acceleration estimate}\label{sec:accintro}
With estimate \eqref{4} at our disposal we return to the fluid-structure interaction problem \eqref{1}--\eqref{2}. We
aim at testing the structure equation \eqref{2} with $\partial_t^2\eta$ and seek for an appropriate test-function for the momentum equation. Due to the condition $\bfu\circ\bfvarphi_\eta=\partial_t\eta\bfn$ at the boundary we have 
\begin{align*}
\frac{\D}{\D t}\bfu&=\partial_t\bfu(t,x+\eta\bfn)+(\bfu(t,x+\eta\bfn)\cdot\nabla)\bfu(t,x+\eta\bfn)\\
&=\partial_t\bfu(t,x+\eta\bfn)+\partial_t\eta\bfn\cdot\nabla\bfu(t,x+\eta\bfn)\quad \text{on}\quad \partial\Omega_\eta.
\end{align*}
Hence the material derivative of the velocity field is the corresponding test-function for the momentum equation. To be precise we use 
\begin{align}\label{eq:materialtest}\partial_t\bfu+(\mathscr F_\eta(\partial_t\eta\bfn)\cdot\nabla)\bfu,
\end{align}
where $\mathscr F_\eta$ is an appropriate extension operator, see Section \ref{sec:ext}. The drawback with the function in \eqref{eq:materialtest} is that it is not solenoidal. This problem is overcome in \cite{GraHil} by using instead
\begin{align}\label{eq:materialtest'}\partial_t\bfu+(\mathscr F_\eta(\partial_t\eta\bfn)\cdot\nabla)\bfu-(\bfu\cdot\nabla )\mathscr F_\eta(\partial_t\eta\bfn).
\end{align}
In fact, in the simplified geometric set-up used in \cite{GraHil} it is possible to construct an extension which is at the same time solenoidal and satisfies
\begin{align}\label{on}
\nabla \mathscr F_\eta(\partial_t\eta\bfn)\cdot \bfn=0\quad \text{on}\quad \partial\Omega_\eta.
\end{align}
In conclusion, the function in \eqref{eq:materialtest'} is solenoidal and equals to $\partial_t^2\eta$ at the boundary.
In general case, some elementary calculations based on Fourier expansion reveal that both conditions cannot hold simultaneous. Therefor, the existing extension operators from \cite{LeRu} and \cite{MuSc} are solenoidal but do not satisfy \eqref{on}. In Section \ref{sec:accest} we propose an alternative approach which is based on \eqref{eq:materialtest} and an elementary extension operator introduced in Section
\ref{sec:ext}. It is not solenoidal but, different to those from  \cite{LeRu} and \cite{MuSc}, has the usual regularisation property (which is inverse to the trace embedding from \eqref{eq:traceembedding}). Accordingly, me must introduce the pressure function and estimate it. This can be done with the help of \eqref{4}, see the proof of Proposition \ref{prop2} for details.

In order to implement this idea rigorously we first prove the existence of a local-in-time strong solution in Section  \ref{sec:strong}. As in previous papers, where the flat geometry is considered, we follow a standard approach based on a transformation of \eqref{1} to the reference geometry, linearisation and a fixed point argument. Our situation is, however, technically more complicated due to the non-trivial transformation map between the reference and moving geometry.

\section{Preliminaries}
\subsection{Conventions}
For notational simplicity we set all physical constants in \eqref{1}--\eqref{2} to 1. The analysis is not effected as long as they are strictly positive.
We write $f\lesssim g$ for two non-negative quantities $f$ and $g$ if there is a $c>0$ such that $f\leq\,c g$. Here $c$ is a generic constant which does not depend on the crucial quantities. If necessary we specify particular dependencies. We write $f\approx g$ if $f\lesssim g$ and $g\lesssim f$.
We do not distinguish in the notation for the function spaces between scalar- and vector-valued functions. However, vector-valued functions will usually be denoted in bold case.
For simplicity we supplement \eqref{2} with periodic boundary conditions and identify $\omega$ (which represents the complete boundary of $\Omega$) with the interval $(0,1)$. We consider periodic function spaces for zero-average functions.
It is only a technical matter to consider instead \eqref{2} on a nontrivial subset of $\partial\Omega$ together with zero boundary conditions for $\eta$ and $\partial_y\eta$, see, e.g., \cite{LeRu} or \cite{BrSc} for the corresponding geometrical set-up.
We shorten the time interval $(0,T)$ by $I$.

\subsection{Classical function spaces}
Let $\mathcal O\subset\R^m$, $m\geq 1$, be open.
Function spaces of continuous or $\alpha$-H\"older-continuous functions, $\alpha\in(0,1)$,
 are denoted by $C(\overline{\mathcal O})$ or $C^{0,\alpha}(\overline{\mathcal O})$ respectively. Similarly, we write $C^1(\overline{\mathcal O})$ and $C^{1,\alpha}(\overline{\mathcal O})$.
We denote as usual by $L^p(\mathcal O)$ and $W^{k,p}(\mathcal O)$ for $p\in[1,\infty]$ and $k\in\mathbb N$ Lebesgue and Sobolev spaces over $\mathcal O$. For a bounded domain $\mathcal O$ the space $L^p_\perp(\mathcal O)$ denotes the subspace of  $L^p(\mathcal O)$ of functions with zero mean, that is $(f)_{\mathcal O}:=\dashint_{\mathcal O}f\dx:=\mathcal L^m(\mathcal O)^{-1}\int_{\mathcal O}f\dx=0$.
 We denote by $W^{k,p}_0(\mathcal O)$ the closure of the smooth and compactly supported functions in $W^{k,p}(\mathcal O)$. If $\partial\mathcal O$ is regular enough, this coincides with the functions vanishing $\mathcal H^{m-1}$ -a.e. on $\partial\mathcal O$. 
 We also denote by $W^{-k,p}(\mathcal O)$ the dual of $W^{k,p}_0(\mathcal O)$.
  Finally, we consider subspaces
$W^{1,p}_{\Div}(\mathcal O)$ and $W^{1,p}_{0,\Div}(\mathcal O)$ of divergence-free vector fields which are defined accordingly. The space $L^p_{\Div}(\mathcal O)$ is defined as the closure of the smooth and compactly supported solenoidal functions in $L^p(\mathcal O).$ We will use the shorthand notations $L^p_x$ and $W^{k,p}_x$ in the case of $n$-dimensional domains (typically spaces defined over $\Omega\subset\R^n$ or $\Omega_\eta\subset\R^n$) and   
$L^p_y$ and $W^{k,p}_y$ for $(n-1)$ dimensional sets (typcially spaces of periodic functions defined over $\omega\subset\R$). 

For a separable Banach space $(X,\|\cdot\|_X)$ we denote by $L^p(0,T;X)$ the set of (Bochner-) measurable functions $u:(0,T)\rightarrow X$ such that the mapping $t\mapsto \|u(t)\|_{X}\in L^p(0,T)$. 
The set $C([0,T];X)$ denotes the space of functions $u:[0,T]\rightarrow X$ which are continuous with respect to the norm topology on $(X,\|\cdot\|_X)$. For $\alpha\in(0,1]$ we write
$C^{0,\alpha}([0,T];X)$ for the space of H\"older-continuous functions with values in $X$. The space $W^{1,p}(0,T;X)$ consists of those functions from $L^p(0,T;X)$ for which the distributional time derivative belongs to $L^p(0,T;X)$ as well. The space $W^{k,p}(0,T;X)$ is defined accordingly.
We use the shorthand $L^p_tX$ for $L^p(0,T;X)$. For instance, we write $L^p_tW^{1,p}_x$ for $L^p(0,T;W^{1,p}(\mathcal O))$. Similarly, $W^{k,p}_tX$ stands for $W^{k,p}(0,T;X)$.

\subsection{Fractional differentiability and Sobolev mulitpliers}
For $p\in[1,\infty)$ the fractional Sobolev space (Sobolev-Slobodeckij space) with differentiability $s>0$ with $s\notin\mathbb N$ will be denoted by $W^{s,p}(\mathcal O)$. For $s>0$ we write $s=\lfloor s\rfloor+\lbrace s\rbrace$ with $\lfloor s\rfloor\in\N_0$ and $\lbrace s\rbrace\in(0,1)$.
 We denote by $W^{s,p}_0(\mathcal O)$ the closure of the smooth and compactly supported functions in $W^{1,p}(\mathcal O)$. For $s>\frac{1}{p}$ this coincides with the functions vanishing $\mathcal H^{m-1}$ -a.e. on $\partial\mathcal O$ provided $\partial\mathcal O$ is regular enough. We also denote by $W^{-s,p}(\mathcal O)$ for $s>0$ the dual of $W^{s,p}_0(\mathcal O)$. Similar to the case of unbroken differentiabilities above we use the shorthand notations $W^{s,p}_x$  and $W^{s,p}_y$. 
We will denote by $\bfB^s_{p,q}(\R^m)$ the standard Besov spaces on $\R^m$ with differentiability $s>0$, integrability $p\in[1,\infty]$ and fine index $q\in[1,\infty]$. They can be defined (for instance) via Littlewood-Paley decomposition leading to the norm $\|\cdot\|_{\bfB^s_{p,q}(\R^m)}$. 
 We refer to \cite{RuSi} and \cite{Tr,Tr2} for an extensive picture. 
 The Besov spaces $\bfB^s_{p,q}(\mathcal O)$ for a bounded domain $\mathcal O\subset\R^m$ are defined as the restriction of functions from $\bfB^s_{p,q}(\R^m)$, that is
 \begin{align*}
 \bfB^s_{p,q}(\mathcal O)&:=\{f|_{\mathcal O}:\,f\in \bfB^s_{p,q}(\R^m)\},\\
 \|g\|_{\bfB^s_{p,q}(\mathcal O)}&:=\inf\{ \|f\|_{\bfB^s_{p,q}(\R^m)}:\,f|_{\mathcal O}=g\}.
 \end{align*}
 If $s\notin\mathbb N$ and $p\in(1,\infty)$ we have $\bfB^s_{p,p}(\mathcal O)=W^{s,p}(\mathcal O)$.
 
In accordance with \cite[Chapter 14]{MaSh} the Sobolev multiplier norm  is given by
\begin{align}\label{eq:SoMo}
\|\varphi\|_{\mathcal M(W^{s,p}(\mathcal O))}:=\sup_{\bfv:\,\|\bfv\|_{W^{s-1,p}(\mathcal O)}=1}\|\nabla\varphi\cdot\bfv\|_{W^{s-1,p}(\mathcal O)},
\end{align}
where $p\in[1,\infty]$ and $s\geq1$.
The space $\mathcal M(W^{s,p}(\mathcal O))$ of Sobolev multipliers is defined as those objects for which the $\mathcal M(W^{s,p}(\mathcal O))$-norm is finite. By mathematical induction with respect to $s$ one can prove for Lipschitz-continuous functions $\varphi$ that membership to $\mathcal M(W^{s,p}(\mathcal O))$  in the sense of \eqref{eq:SoMo} implies that
\begin{align}\label{eq:SoMo'}
\sup_{w:\,\|w\|_{W^{s,p}(\mathcal O)}=1}\|\varphi \,w\|_{W^{s,p}(\mathcal O)}<\infty.
\end{align}
The quantity \eqref{eq:SoMo'} also serves as customary definition of the Sobolev multiplier norm in the literature but \eqref{eq:SoMo} is more suitable for our purposes.
Note that in our applications we always assume that the functions in question are Lipschitz continuous such that the implication above is given. 

Let us finally collect some some useful properties of Sobolev multipliers.
By \cite[Corollary 14.6.2]{MaSh} we have
\begin{align}\label{eq:MSa}
\|\phi\|_{\mathcal M(W^{s,p}(\R^{m}))}\lesssim\|\nabla\phi\|_{L^{\infty}(\R^m)},
\end{align}
provided that one of the following conditions holds:
\begin{itemize}
\item $p(s-1)<m$ and $\phi\in \bfB^{s}_{\varrho,p}(\R^{m})$ with $\varrho\in\big[\frac{pm}{p(s-1)-1},\infty\big]$;
\item $p(s-1)=m$ and $\phi\in\bfB^{s}_{\varrho,p}(\R^m)$ with $\varrho\in(p,\infty]$.
\end{itemize}
Note that the hidden constant in \eqref{eq:MSa} depends on the $\bfB^{s}_{\varrho,p}(\R^{m})$-norm of $\phi$.
By \cite[Corollary 4.3.8]{MaSh} it holds
\begin{align}\label{eq:MSb}
\|\phi\|_{\mathcal M(W^{s,p}(\R^{m}))}\approx
\|\nabla\phi\|_{W^{s-1,p}(\R^{m})} 
\end{align}
for $p(s-1)>m$. 
 Finally, we note the following rule about the composition with Sobolev multipliers which is a consequence of \cite[Lemma 9.4.1]{MaSh}. For open sets $\mathcal O_1,\mathcal O_2\subset\R^m$, $u\in W^{s,p}(\mathcal O_2)$ and a Lipschitz continuous function $\bfphi:\mathcal O_1\rightarrow\mathcal O_2$ with Lipschitz continuous inverse and $\bfphi\in \mathcal M(W^{s,p}(\mathcal O_1))$ we have
\begin{align}\label{lem:9.4.1}
\|u\circ\bfphi\|_{W^{s,p}(\mathcal O_1)}\lesssim \|u\|_{W^{s,p}(\mathcal O_2)}
\end{align}
with constant depending on $\bfphi$. Using Lipschitz continuity
of $\bfphi$ and $\bfphi^{-1}$, estimate \eqref{lem:9.4.1} is obvious for $s\in(0,1]$. The general case can be proved by mathematical induction with respect to $s$.

\subsection{Function spaces on variable domains}
\label{ssec:geom}
 The spatial domain $\Omega$ is assumed to be an open bounded subset of $\mathbb{R}^n$, $n=2,3$, with smooth boundary and an outer unit normal ${\bfn}$. Wee assume that
 $\partial\Omega$ can be parametrised by an injective mapping ${\bfvarphi}\in C^k(\omega;\R^n)$ for some sufficiently large $k\in\N$. If $n=3$ we suppose for all points $y=(y_1,y_2)\in \omega$ that the pair of vectors  
$\partial_i {\bfvarphi}(y)$, $i=1,2,$ are linearly independent.
If $n=2$ the corresponding assumption simply asks for $\partial_y\bfvarphi$ not to vanish.
 For a point $x$ in the neighborhood
or $\partial\Omega$ we can define the functions $y$ and $s$ by
\begin{align*}
 y(x)=\arg\min_{y\in\omega}|x-\bfvarphi(y)|,\quad s(x)=(x-y(x))\cdot\bfn(y(x)).
 \end{align*}
Moreover, we define the projection $\bfp(x)=\bfvarphi(y(x))$. We define $L>0$ to be the largest number such that $s,y$ and $\bfp$ are well-defined on $S_L$, where
\begin{align}
\label{eq:boundary1}
S_L=\{x\in\R^n:\,\mathrm{dist}(x,\partial\Omega)<L\}.
\end{align}
Due to the smoothness of $\partial\Omega$ for $L$ small enough we have $\abs{s(x)}=\min_{y\in\omega}|x-\bfvarphi(y)|$ for all $x\in S_L$. This implies that $S_L=\{s\bfn(y)+y:(s,y)\in (-L,L)\times \omega\}$.
%
For a given function $\eta : I \times \omega \rightarrow\R$ we parametrise the deformed boundary by
\begin{align*}
{\bfvarphi}_\eta(t,y)={\bfvarphi}(y) + \eta(t,y){\bfn}(y), \quad \,y \in \omega,\,t\in I.
\end{align*}
By possibly decreasing $L$, one easily deduces from this formula that $\Omega_{\eta}$ does not degenerate, that is
\begin{align}\label{eq:1705}
\begin{aligned}
\partial_y\bfvarphi_\eta(t,y)\neq0\quad\text{if} \quad n&=2,\quad\partial_1\bfvarphi_\eta\times\partial_2\bfvarphi_\eta(t,y)\neq0\quad\text{if} \quad n=3,\\
 \bfn(y)\cdot\bfn_{\eta(t)}(y)&>0,\quad \,y \in \omega,\,t\in I,
 \end{aligned}
\end{align}
provided $\|\eta\|_{L^\infty_{t,x}}<L$. Here $\bfn_{\eta(t)}$
is the normal of the domain $\Omega_{\eta(t)}$
 defined through
\begin{align}\label{eq:2612}
\partial\Omega_{\eta(t)}=\set{{\bfvarphi}(y) + \eta(t,y){\bfn}(y):y\in \omega}.
\end{align}
With some abuse of notation we define deformed space-time cylinder $I\times\Omega_\eta=\bigcup_{t\in I}\set{t}\times\Omega_{\eta(t)}\subset\R^{1+n}$.
The corresponding function spaces for variable domains are defined as follows.
\begin{definition}{(Function spaces)}
For $I=(0,T)$, $T>0$, and $\eta\in C(\overline{I}\times\omega)$ with $\|\eta\|_{L^\infty(I\times\omega)}< L$ we define for $1\leq p,r\leq\infty$
\begin{align*}
L^p(I;L^r(\Omega_\eta))&:=\big\{v\in L^1(I\times\Omega_\eta):\,\,v(t,\cdot)\in L^r(\Omega_{\eta(t)})\,\,\text{for a.e. }t,\,\,\|v(t,\cdot)\|_{L^r(\Omega_{\eta(t)})}\in L^p(I)\big\},\\
L^p(I;W^{1,r}(\Omega_\eta))&:=\big\{v\in L^p(I;L^r(\Omega_\eta)):\,\,\nabla v\in L^p(I;L^r(\Omega_\eta))\big\}.
\end{align*}
\end{definition}
For various purposes it is useful to relate the time dependent domain and the fixed domain. This can be done by the means of the Hanzawa transform. Its construction can be found in
\cite[pages 210, 211]{LeRu}. Note that variable domains in \cite{LeRu} are defined via functions $\zeta:\partial\Omega\rightarrow\R$ rather than functions $\eta:\omega\rightarrow\R$ (clearly, one can link them by setting
$\zeta=\eta\circ\bfvarphi^{-1}$). For any  $\eta:\omega\rightarrow(-L,L)$ we define the Hanzawa transform $\bfPsi_\eta: \Omega  \to \Omega_\eta$
by
\begin{align}
\label{map}
\begin{aligned}
\bfPsi_\eta(x)&= \begin{cases}\bfp(x)+\Big(s(x)+\eta(y(x))\phi(s(x))\Big)\bfn(y(x)),&\text{ if } \dist(x,\partial\Omega)<L,\\
\quad x,\quad &\text{elsewhere}.
\end{cases}
\end{aligned}
\end{align}
Here $\phi\in C^\infty(\mathbb R)$ is such that 
$\phi\equiv 0$ in neighborhood of $-L$ and $\phi\equiv 1$ in a neighborhood of $0$.
Due to the size of $L$, we find that $\bfPsi_\eta$ is a homomorphism such that $\bfPsi_\eta|_{\Omega\setminus S_L}$ is the identity. We clearly have for $k\in\N$ and $p\in[1,\infty]$
\begin{align}\label{est:psieta}
\|\bfPsi_\eta\|_{W^{k,p}_x}\lesssim 1+\|\eta\|_{W^{k,p}_y},\quad \eta\in W^{k,p}(\omega),
\end{align}
as well as
\begin{align}\label{est:psietazeta}
\|\bfPsi_\eta-\bfPsi_\zeta\|_{W^{k,p}_x}\lesssim \|\eta-\zeta\|_{W^{k,p}_y},\quad \eta,\zeta\in W^{k,p}(\omega),
\end{align}
where the hidden constant only depends on the reference geometry. 

If $\|\eta\|_{L^\infty_y}<\alpha<L$ and $\|\nabla\eta\|_{L^\infty_y}< R$ for some $\alpha,R>0$
the inverse\footnote{It exists provided we choose $\phi$ such that $|\phi'|<L/\alpha$.} $\bfPsi_\eta^{-1}:\Omega_\eta\rightarrow\Omega$ satisfies for $k\in\N$ and $p\in[1,\infty]$
\begin{align}\label{est:psi-1eta}
\|\bfPsi_\eta^{-1}\|_{W^{k,p}_x}\lesssim 1+\|\eta\|_{W^{k,p}_y},\quad \eta\in W^{k,p}(\omega),
\end{align}
as well as
\begin{align}\label{est:psi-1etazeta}
\|\bfPsi_\eta^{-1}-\bfPsi^{-1}_\zeta\|_{W^{k,p}_x}\lesssim \|\eta-\zeta\|_{W^{k,p}_y},\quad \eta,\zeta\in W^{k,p}(\omega),
\end{align}
if $\|\zeta\|_{L^\infty_y}<\alpha$ and $\|\nabla\zeta\|_{L^\infty_y}< R$.
In \eqref{est:psi-1eta} and \eqref{est:psi-1etazeta} the hidden constant depends on the reference geometry (which is assumed to be sufficiently smooth), on $L-\alpha$ and $R$. Similarly, we obtain fractional estimates, that is
\begin{align}
\label{est:psietas}
\|\bfPsi_\eta\|_{W^{s,p}_x}\lesssim 1+\|\eta\|_{W^{s,p}_y},\quad \eta\in W^{s,p}(\omega),\\
\label{est:psi-1etas}
\|\bfPsi_\eta^{-1}\|_{W^{s,p}_x}\lesssim 1+\|\eta\|_{W^{s,p}_y},\quad \eta\in W^{s,p}(\omega),
\end{align}
for $s>0 $ with $s\notin\N$ and
\begin{align}\label{est:psietazetas}
\|\bfPsi_\eta-\bfPsi_\zeta\|_{W^{s,p}_x}\lesssim \|\eta-\zeta\|_{W^{k,p}_y},\quad \eta,\zeta\in W^{s,p}(\omega),\\
\label{est:psi-1etazetas}
\|\bfPsi_\eta^{-1}-\bfPsi^{-1}_\zeta\|_{W^{s,p}_x}\lesssim \|\eta-\zeta\|_{W^{s,p}_y},\quad \eta,\zeta\in W^{s,p}(\omega).
\end{align}
 Finally, it holds
\begin{align}
\label{est:psietast}
\|\partial_t\bfPsi_\eta\|_{W^{s,p}_x}\lesssim 1+\|\partial_t\eta\|_{W^{s,p}_y},\quad \eta\in W^{1,1}(I;W^{s,p}(\omega)),
\end{align}
 uniformly in time.

\subsection{Extensions on variable domains}\label{sec:ext}
In this subsection we construct an extension operator
which extends functions from $\omega$ to the moving domain $\Omega_\eta$ for a given function $\eta$ defined on $\omega$. We follow \cite[Section 2.3]{BrSc}.
Since $\Omega$ is assumed to be sufficiently smooth, it is well-known that there is an extension operator $\mathscr F_\Omega$ which extends functions from $\partial\Omega$ to $\R^n$ and satisfies
\begin{align}\label{2.17a}
\mathscr F_\Omega:W^{\sigma,p}(\partial\Omega)\rightarrow W^{\sigma+1/p,p}(\R^n)
\end{align}
for all $p\in[1,\infty]$ and $\sigma\geq0$, all as well as $\mathscr F_\Omega v|_{\partial\Omega}=v$. Now we define $\mathscr F_\eta$ by 
\begin{align}\label{eq:2401b}
\mathscr F_\eta b=\mathscr F_\Omega ((b\bfn)\circ\bfvarphi^{-1})\circ{\bfPsi}_\eta^{-1},\quad b\in W^{\sigma,p}(\omega),
\end{align}
where $\bfvarphi$ is the function in the parametrisation of $\Omega$.
If $\eta$ is regular enough, $\mathscr F_\eta$ behaves as a classical extension. To be more precise, we can use the formula
\begin{align*}
\nabla\mathscr F_\eta b&=\nabla\mathscr F_\Omega ((b\bfn)\circ\bfvarphi^{-1})\circ{\bfPsi}_\eta^{-1}\nabla{\bfPsi}_\eta^{-1},
\end{align*}
estimate \eqref{est:psi-1eta} and \eqref{2.17a} to obtain the following.
\begin{lemma}\label{lem:3.8}
Let $\eta\in C^{0,1}(\omega)$ with $\|\eta\|_{L^\infty_y}<\alpha<L$. 
 The operator
$\mathscr F_\eta$ defined in \eqref{eq:2401b} satisfies for all $p\in(1,\infty]$, \footnote{It is possible to obtain a theory for any $\sigma\geq0$ provided $\eta$ is sufficiently regular.} $\sigma\in(0,1-\tfrac{1}{p}]$ and $s\in(0,\tfrac{1}{p})$,
\begin{align*}
\mathscr F_\eta: W^{\sigma,p}(\omega)\rightarrow W^{\sigma+1/p,p}(\Omega\cup S_\alpha),\quad \mathscr F_\eta: L^{p}(\omega)\rightarrow W^{s,p}(\Omega\cup S_\alpha)
\end{align*}
and $(\mathscr F_\eta b)\circ\bfvarphi_\eta=b\bfn$ on $\omega$ for all $b\in L^{p}(\omega)$. In particular, we have
\begin{align*}
\|\mathscr F_\eta b\|_{W^{\sigma+1/p,p}(\Omega\cup S_\alpha)}\lesssim\|b\|_{W^{\sigma,p}(\omega)},\quad \|\mathscr F_\eta b\|_{W^{s,p}(\Omega\cup S_\alpha)}\lesssim\|b\|_{L^{p}(\omega)},
\end{align*}
where the hidden constant depends only on $\Omega,p,\sigma$, $\|\nabla\eta\|_{L^\infty_y}$ and $L-\alpha$.
\end{lemma}

\subsection{The concept of solutions and the main results}
In this section we introduce the framework for the system \eqref{1}--\eqref{2} and present our main results concerning the regularity of solutions. We start with the definition of a weak solution.
Note that different to the previous subsection we assume again that $n=2$.
\begin{definition}[Weak solution] \label{def:weakSolution}
Let $(\bff, g, \eta_0,  \bfu_0, \eta_1)$ be a dataset such that
\begin{equation}
\begin{aligned}
\label{dataset}
&\bff \in L^2\big(I; L^2_{\mathrm{loc}}(\mathbb{R}^2)\big),\quad
g \in L^2\big(I; L^2(\omega)\big), \quad
\eta_0 \in W^{2,2}(\omega) \text{ with } \Vert \eta_0 \Vert_{L^\infty( \omega)} < L, 
\\
& 
\bfu_0\in L^2_{\mathrm{\Div}}(\Omega_{\eta_0}) \text{ is such that }\bfu_0\circ\bfvarphi_{\eta_0} =\eta_1 \bfn\text{ on $\omega$}, \quad
\eta_1 \in L^2(\omega).
\end{aligned}
\end{equation} 
We call the tuple
$(\eta,\bfu)$
a weak solution to the system \eqref{1}--\eqref{2} with data $(\bff, g, \eta_0,  \eta_1,\bfu_0)$ provided that the following holds:
\begin{itemize}
\item[(a)] The structure displacement $\eta$ satisfies
\begin{align*}
\eta \in W^{1,\infty} \big(I; L^2(\omega) \big)\cap W^{1,2} \big(I; W^{1,2}(\omega) \big)\cap  L^\infty \big(I; W^{2,2}(\omega) \big) \quad \text{with} \quad \Vert \eta \Vert_{L^\infty(I \times \omega)} <L,
\end{align*}
as well as $\eta(0)=\eta_0$ and $\partial_t\eta(0)=\eta_1$.
\item[(b)] The velocity field $\bfu$ satisfies
\begin{align*}
 \bfu \in L^\infty \big(I; L^2(\Omega_{\eta}) \big)\cap  L^2 \big(I; W^{1,2}_{\Div}(\Omega_{\eta}) \big) \quad \text{with} \quad 
\bfu\circ\bfvarphi_\eta =\partial_t \eta {\bfn}\quad\text{on}\quad I\times\omega,
\end{align*}
as well as $\bfu(0)=\bfu_0$.
\item[(c)] For all  $(\phi, {\bfphi}) \in C^\infty(\overline{I}\times\omega) \times C^\infty_{\Div}(\overline{I}\times \R^2; \R^2)$ with $\phi(T,\cdot)=0$, ${\bfphi}(T,\cdot)=0$ and $\bfphi\circ\bfvarphi_\eta =\phi {\bfn}$ on $I\times\omega$, we have
\begin{align*}
\int_I  \frac{\mathrm{d}}{\dt}\bigg(\int_{\Oeta}\bfu  \cdot {\bfphi}\dx
+\int_\omega \partial_t \eta \, \phi \dy
\bigg)\dt 
&=\int_I  \int_{\Oeta}\big(  \bfu\cdot \partial_t  {\bfphi} + \bfu \otimes \bfu: \nabla {\bfphi} 
  \big) \dx\dt
\\&
-\int_I  \int_{\Oeta}\big(   
\nabla \bfu:\nabla {\bfphi} -\bff\cdot{\bfphi} \big) \dx\dt\\
&+
\int_I \int_\omega \big(\partial_t \eta\, \partial_t\phi-\partial_t\partial_y\eta\,\partial_y \phi-
 g\, \phi \big)\dy\dt\\&-\int_I\int_\omega \partial_y^2\eta\,\partial_y^2 \phi\dy\dt.
\end{align*}
\end{itemize}
\end{definition}
The existence of a weak solution can be shown as in \cite{LeRu}. The term $\partial_t\partial_y^2\eta$ is not included there, but it does not alter the arguments. 
Note that we use a pressure-free formulation (that is, with test-function satisfying additionally $\Div\bfphi=0$) here. If the solution possess more regularity,
the pressure can be recovered by setting
\begin{align*}
\tilde \pi_0:=\Delta_{\eta}^{-1}\Div((\nabla\bfu)\bfu),\quad \pi_0:=\tilde\pi_0-(\tilde\pi_0)_{\Omega_{\eta}}.
\end{align*}
For $\mathcal O\subset\R^n$ open and bounded with normal $\bfn_{\mathcal O}$ 
we denote by $\Delta^{-1}_{\mathcal O}\Div$ the solution operator to the equation
\begin{align*}
\Delta h=\Div\bfg\quad\text{in}\quad\mathcal O,\quad \bfn_{\mathcal O}\cdot(\nabla h-\bfg)=0\quad\text{on}\quad\partial\mathcal O.
\end{align*}
We must complement $\pi_0$ by a function depending on time only being uniquely determined by the structure equation. Setting
$\pi(t)=\pi_0(t)+c_\pi(t)$ and 
testing the structure equation with 1 we obtain
\begin{align}\label{eq:pressure}
c_\pi(t)\int_{\omega}\bfn\cdot\bfn_\eta|\partial_y\bfvarphi_\eta|\dy=\int_{\omega}\bfn\big(\nabla\bfu+\nabla\bfu^\top-\pi_0\mathbb I_{2\times 2}\big)\circ\bfvarphi_\eta\bfn_\eta|\partial_y\bfvarphi_\eta|\dy+\int_\omega\partial_t^2\eta\dy-\int_\omega g\dy
\end{align}
 Since $\Omega_\eta$ is Lipschitz uniformly in time the operator $\Delta_{\Omega\eta}^{-1}\Div$ has the usual properties. In particular, it is continuous $L^2\rightarrow W^{1,2}$ such that
\begin{align*}
\int_I\int_{\Omega_\eta}|\nabla\pi|^2\dxt\lesssim \int_I\int_{\Omega_\eta}|(\nabla\bfu)\bfu)|^2\dxt\lesssim\bigg( \int_I\|\bfu\|^4_{L^4(\Omega_\eta)}\bigg)^{\frac{1}{2}}\bigg(\int_I\|\nabla\bfu\|^4_{L^{4}(\Omega_\eta)}\dt\bigg)^{\frac{1}{2}}
\end{align*}
by Ladyshenskaya's inequality (using again that $\Omega_\eta$ is Lipschitz uniformly in time). Hence we have $\pi\in L^2(I\;W^{1,2}(\Omega_\eta))$ provided the right-hand side is finite (which is the case if $\bfu$ and $\nabla\bfu$ belong to $L^4$ in space-time). This is the case for a strong solution which is defined as follows.

\begin{definition}[Strong solution] \label{def:strongSolution}
We call the triple $(\eta,\bfu,\pi)$ a strong solution to \eqref{1}--\eqref{2} provided $(\eta,\bfu)$ is a weak solution to \eqref{1}--\eqref{2}, it satisfies
\begin{align*}
\eta \in W^{1,\infty} \big(I; W^{1,2}(\omega) \big)\cap W^{1,2} \big(I; W^{2,2}(\omega) \big)\cap  L^\infty \big(I; W^{3,2}(\omega) \big) \cap W^{2,2}(I;L^2(\omega)),\\
 \bfu \in W^{1,2} \big(I; L^{2}(\Omega_{\eta}) \big)\cap  L^2 \big(I; W^{2,2}(\Omega_{\eta}) \big),\quad\pi\in  L^2 \big(I; W^{1,2}(\Omega_{\eta}) \big),
\end{align*}
and we have $\nabla\pi=\nabla\Delta_{\Omega_\eta}^{-1}\Div((\nabla\bfu)\bfu)$.
\end{definition}
For a strong solution $(\eta,\bfu,\pi)$ the momentum equation holds in the strong sense, that is we have
 \begin{align} \partial_t \bfu+(\nabla\bfu)\bfu&=\Delta\bfu-\nabla \pi+\bff&\label{1'}
 \end{align}
a.a. in $I\times\Omega_\eta$. The beam equation together with the regularity properties above yield $\eta\in L^2(I;W^{4,2}(\omega))$.  Hence the beam equation holds in the strong sense as well, that is we have
 \begin{align}\label{2'}
\ \partial_t^2\eta-\partial_t\partial_y^2\eta+\partial_y^4\eta=g-\bfn \bftau\circ\bfvarphi_\eta\partial_y\bfvarphi_\eta^\perp
 \end{align}
 a.a. in $I\times\omega$. 
Note that for a strong solution the Cauchy stress $\bftau=\nabla\bfu+\nabla\bfu^\top-\pi\mathbb I_{2\times 2}$ possesses enough regularity to be evaluated at the moving boundary (this is due to the trace theorem and the uniform Lipschitz continuity of $\Omega_\eta$).

We are finally ready to state our main result concerning the existence of a unique strong solution to \eqref{1}--\eqref{2}.
\begin{theorem}\label{thm:main}
Suppose that the dataset $(\bff, g, \eta_0,  \bfu_0, \eta_1)$
satisfies in addition to \eqref{dataset} that
\begin{equation}
\begin{aligned}
\label{dataset'}
g \in L^2\big(I; W^{1,2}(\omega)\big), \quad
\eta_0 \in W^{3,2}(\omega) , \quad \eta_1 \in W^{1,2}(\omega),\quad 
\bfu_0\in W^{1,2}_{\mathrm{\Div}}(\Omega_{\eta_0}).
\end{aligned}
\end{equation}
Then there is a unique strong solution to \eqref{1}--\eqref{2} in the sense of Definition \ref{def:strongSolution}. The interval of existence is of the form $I = (0, t)$, where $t < T$ only in case $\Omega_{\eta(s)}$ approaches a self-intersection when $s\rightarrow t$ or it degenerates\footnote{Self-intersection and degeneracy are excluded if $\|\eta\|_{L^\infty_{t,x}}<L$, cf. \eqref{eq:boundary1} and \eqref{eq:1705}.} (namely, if $\lim_{s\rightarrow t}\partial_y\bfvarphi_\eta(s,\omega)=0$ or $\lim_{s\rightarrow t}\bfn(y)\cdot\bfn_{\eta(s)}(y)=0$ for some $y\in\omega$).
\end{theorem}
The proof of Theorem \ref{thm:main} can be found in Section \ref{sec:reg}.

\begin{remark}
A drawback of Theorem \ref{thm:main} compared to the corresponding statement for the flat geometry from \cite{GraHil} is that we can currently not exclude a self-intersection of the moving domain for arbitrary times. It would be a of great interest to prove a distance estimate as in \cite[Section 4.2]{GraHil}
in the present set-up. 
\end{remark}

\section{The Stokes equations in non-smooth domains}
This section is devoted to the study of the Stokes equations
in a domain $\mathcal O\subset\R^n$, $n\geq 2$, with minimal regularity.
We start by introducing the necessary framework to parametrise
the boundary of the underlying domain by local maps of a certain regularity. This yields, in particular, a rigorous definition of a  $\bfB^s_{\rho,q}$-boundary. In Section \ref{sec:stokessteady} we consider the steady Stokes system. This will be crucial for the acceleration estimate for the fluid-structure problem in Section \ref{sec:accest} (we explain in Remark \ref{rem:stokes} how to parametrise the sets $\Omega_\eta$ introduced in Section \ref{ssec:geom} by local maps). 
 
\subsection{Parametrisation of domains}\label{sec:para}
 Let $\mathcal O\subset\R^n$, $n\geq 2$, by a bounded open set.
We assume that $\partial{\mathcal{O}}$ can be covered by a finite
number of open sets $\mathcal U^1,\dots,\mathcal U^\ell$ for some $\ell\in\mathbb N$, such that
the following holds. For each $j\in\{1,\dots,\ell\}$ there is a reference point
$y^j\in\R^n$ and a local coordinate system $\{e^j_1,\dots,e_n^j\}$ (which we assume
to be orthonormal and set $\mathcal Q_j=(e_1^j|\dots |e_n^j)\in\mathbb R^{n\times n}$), a function
$\varphi_j:\mathbb R^{n-1}\rightarrow\mathbb R$
and $r_j>0$
with the following properties:
\begin{enumerate}[label={\bf (A\arabic{*})}]
\item\label{A1} There is $h_j>0$ such that
$$\mathcal U^j=\{x=\mathcal Q_jz+y^j\in\mathbb R^n:\,z=(z',z_n)\in\R^n,\,|z'|<r_j,\,
|z_n-\varphi_j(z')|<h_j\}.$$
\item\label{A2} For $x\in\mathcal U^j$ we have with $z=\mathcal Q_j^\top(x-y^j)$
\begin{itemize}
\item $x\in\partial{\mathcal{O}}$ if and only if $z_n=\varphi_j(z')$;
\item $x\in{\mathcal{O}}$ if and only if $0<z_n-\varphi_j(z')<h_j$;
\item $x\notin{\mathcal{O}}$ if and only if $0>z_n-\varphi_j(z')>-h_j$.
\end{itemize}
\item\label{A3} We have that
$$\partial{\mathcal{O}}\subset \bigcup_{j=1}^\ell\mathcal U^j.$$
\end{enumerate}
In other words, for any $x_0\in\partial{\mathcal{O}}$ there is a neighborhood $U$ of $x_0$ and a function $\varphi:\mathbb R^{n-1}\rightarrow\mathbb R$ such that after translation and rotation\footnote{By translation via $y_j$ and rotation via $\mathcal Q_j$ we can assume that $x_0=0$ and that the outer normal at~$x_0$ is pointing in the negative $x_n$-direction.}
 \begin{align}\label{eq:3009}
 U \cap {\mathcal{O}} = U \cap G,\quad G = \set{(x',x_n)\in \R^n \,:\, x' \in \R^{n-1}, x_n > \varphi(x')}.
 \end{align}
 The regularity of $\partial{\mathcal{O}}$ will be described by means of local coordinates as just described.
 \begin{definition}\label{def:besovboundary}
 Let ${\mathcal{O}}\subset\R^n$ be a bounded domain, $s>0$ and $1\leq \rho,q\leq\infty$. We say that $\partial{\mathcal{O}}$ belongs to the class $\bfB^s_{\rho,q}$ if there is $\ell\in\mathbb N$ and functions $\varphi_1,\dots,\varphi_\ell\in\bfB^s_{\rho,q}(\mathbb R^{n-1})$ satisfying \ref{A1}--\ref{A3}.
 \end{definition}
Clearly, a similar definition applies for a Lipschitz boundary (or a $C^{1,\alpha}$-boundary with $\alpha\in(0,1)$) by requiring that $\varphi_1,\dots,\varphi_\ell\in W^{1,\infty}(\mathbb R^{n-1})$ (or $\varphi_1,\dots,\varphi_\ell\in C^{1,\alpha}(\mathbb R^{n-1})$). We say that the local Lipschitz constant of $\partial{\mathcal{O}}$, denoted by $\mathrm{Lip}(\partial{\mathcal{O}})$, is (smaller or) equal to some number $L>0$ provided the Lipschitz constants of $\varphi_1,\dots,\varphi_\ell$ are not exceeding $L$. Our main result depends on the assumption of a sufficiently small local Lipschitz constant. While this seems rather restrictive at first glance, it appears quite natural when looking closer. Indeed, it holds, for instance, if the regularity
of $\partial{\mathcal{O}}$ is better than Lipschitz (such as $C^{1,\alpha}$ for some $\alpha>0$). By means of the transformations $\mathcal Q_j$ introduced above, we can assume that the reference point $y^j$ in question is the origin and that $\nabla\varphi_j(0)=0$. Choosing $r_j$ in \ref{A1} small enough (which can be achieved simply by allowing more sets in the cover $\mathcal U^1,\dots,\mathcal U^l$) we have
\begin{align*}
|\nabla\varphi_j(z')|=|\nabla\varphi_j(z')-\nabla\varphi_j(0)|\leq\,r_j^\alpha[\nabla\varphi_j]_{C^\alpha}\ll 1
\end{align*}
for all $z'$ with $|z'|\leq r_j$.

In order to describe the behaviour of functions defined in ${\mathcal{O}}$ close to the boundary we need to extend the functions $\varphi_1,\dots,\varphi_\ell$  from \ref{A1}--\ref{A3} to the half space
$\mathbb H := \set{z = (z',z_n)\,:\, z_n > 0}$. Hence we are confronted with the task of extending a function~$\phi\,:\, \R^{n-1}\to \R$ to a mapping $\Phi\,:\, \mathbb H \to \R^n$ that maps the 0-neighborhood in~$\mathbb H$ to the $x_0$-neighborhood in~${\mathcal{O}}$. The mapping $(z',0) \mapsto (z',\phi(z'))$ locally maps the boundary of~$\mathbb H$ to the one of~$\partial {\mathcal{O}}$. We extend this mapping using the extension operator of Maz'ya and Shaposhnikova~\cite[Section 9.4.3]{MaSh}. Let $\zeta \in C^\infty_c(B_1(0'))$ with $\zeta \geq 0$ and $\int_{\R^{n-1}} \zeta(x')\dx'=1$. Let $\zeta_t(x') := t^{-(n-1)} \zeta(x'/t)$ denote the induced family of mollifiers. We define the extension operator 
\begin{align*}
  (\mathcal{T}\phi)(z',z_n)=\int_{\R^{n-1}} \zeta_{z_n}(z'-y')\phi(y')\dy',\quad (z',z_n) \in \mathbb H,
\end{align*}
where~$\phi:\R^n\to \R$ is a Lipschitz function with Lipschitz constant~$K$.
Then the estimate 
\begin{align}\label{est:ext}
  \norm{\nabla (\mathcal{T} \phi)}_{\bfB_{\rho,q}^{s}(\setR^{n})}\le c\norm{\nabla \phi}_{\bfB_{\rho,q}^{s-\frac 1 p}(\setR^{n-1})}
\end{align}
follows from~\cite[Theorem 8.7.2]{MaSh}. Moreover, \cite[Theorem 8.7.2]{MaSh} yields
\begin{align}\label{eq:MS}
\|\mathcal T\phi\|_{\mathcal M(W^{s,p}(\mathbb H))}\lesssim \|\phi\|_{\mathcal M(W^{s-1/p,p}(\R^{n-1}))}.
\end{align}
It is shown in \cite[Lemma 9.4.5]{MaSh} that (for sufficiently large~$N$, i.e., $N \geq c(\zeta) K+1$) the mapping
\begin{align*}
  \alpha_{z'}(z_n) \mapsto N\,z_n+(\mathcal{T} \phi)(z',z_n)
\end{align*}
is for every $z' \in \setR^{n-1}$ one to one and the inverse is Lipschitz with its gradient
bounded by $(N-K)^{-1}$.
Now, we define the mapping~$\bfPhi\,:\, \mathbb H \to \R^n$ as a rescaled version of the latter one by setting
\begin{align}\label{eq:Phi}
  \bfPhi(z',z_n)
  &:=
    \big(z',
    \alpha_{z'}(z_n)\big) = 
    \big(z',
    \,z_n + (\mathcal{T} \phi)(z',z_n/K)\big).
\end{align}
Thus, $\bfPhi$ is one-to-one (for sufficiently large~$N=N(K)$) and we can define its inverse $\bfPsi := \bfPhi^{-1}$.
The mapping $\bfPhi$ has the Jacobi matrix of the form
\begin{align}\label{J}
  J = \nabla \bfPhi = 
  \begin{pmatrix}
    \mathbb I_{(n-1)\times (n-1)}&0
    \\
    \partial_{z'} (\mathcal{T}  \phi)& 1+ 1/N\partial_{z_n}\mathcal{T}  \phi
  \end{pmatrix}.
\end{align}
Since 
$\abs{\partial_{z_n}\mathcal{T}  \phi} \leq K$, we have \begin{align}\label{eq:detJ}\frac{1}{2} < 1-K/N \leq \abs{\det(J)} \leq 1+K/N\leq 2\end{align}
using that $N$ is large compared to~$K$. Finally, we note the implication
\begin{align} \label{eq:SMPhiPsi}
\bfPhi\in\mathcal M(W^{s,p}(\mathbb H))\,\,\Rightarrow \,\,\bfPsi\in\mathcal M(W^{s,p}(\mathbb H)),
\end{align}
which holds, for instance, if $\bfPhi$ is Lipschitz continuous, cf. \cite[Lemma 9.4.2]{MaSh}. In fact, one can prove \eqref{eq:SMPhiPsi} with the help of \eqref{lem:9.4.1} and \eqref{eq:detJ}.

\subsection{The steady Stokes problem}\label{sec:stokessteady}
In this section we consider the steady Stokes system
\begin{align}\label{eq:Stokes}
\Delta \bfu-\nabla\pi=-\bff,\quad\Div\bfu=0,\quad\bfu|_{\partial{\mathcal{O}}}=\bfu_{\partial},
\end{align}
in a domain ${\mathcal{O}}\subset\R^n$ with unit normal $\bfn$. The result given in the following theorem is a maximal regularity estimate for the solution in terms of the right-hand side and the boundary datum under minimal assumption on the regularity of $\partial\mathcal O$. 
\begin{theorem}\label{thm:stokessteady}
Let $p\in(1,\infty)$, $s\geq 1+\frac{1}{p}$ and 
\begin{align}\label{eq:SMp}
\varrho\geq p\quad\text{if}\quad p(s-1)\geq n,\quad \varrho\geq \tfrac{p(n-1)}{p(s-1)-1}\quad\text{if}\quad p(s-1)< n,
\end{align}
 such that 
$n\big(\frac{1}{p}-\frac{1}{2}\big)+1\leq  s$.
 Suppose that ${\mathcal{O}}$ is a $\bfB^{\theta}_{\varrho,p}$-domain for some $\theta>s-1/p$ with locally small Lipschitz constant, $\bff\in W^{s-2,p}({\mathcal{O}})$ and $\bfu_{\partial}\in W^{s-1/p,p}(\partial{\mathcal{O}})$ with $\int_{\partial{\mathcal{O}}}\bfu_\partial\cdot\bfn\,\dd\mathcal H^{n-1}=0$. Then there is a unique solution to \eqref{eq:Stokes} and we have
\begin{align}\label{eq:main}
\|\bfu\|_{W^{s,p}({\mathcal{O}})}+\|\pi\|_{W^{s-1,p}({\mathcal{O}})}\lesssim\|\bff\|_{W^{s-2,p}({\mathcal{O}})}+\|\bfu_{\partial}\|_{W^{s-1/p,p}(\partial{\mathcal{O}})}.
\end{align}
\end{theorem}
\begin{remark}
The theorem holds under the slightly weaker assumption
that $\partial{\mathcal{O}}\in\mathcal M(W^{s-1/p,p})(\delta)$ for $\delta$ sufficiently small. This means that the functions $\varphi_1,\dots,\varphi_\ell$ from the parametrisation of $\partial{\mathcal{O}}$ belong to the the multiplier space $\mathcal M(W^{s,p}(\R^{n-1}))$ with norm
bounded by $\delta$. This is a sharp assumption for the corresponding theory for the Laplace equation, cf. \cite[Chapter 14]{MaSh}. The relationship between $\mathcal M(W^{s-1/p,p})(\delta)$ and Besov spaces can be seen from \eqref{eq:MSa} and \eqref{eq:MSb}. 
\end{remark}
\begin{proof}
By use of a standard extension operator we can assume that $\bfu_\partial=0$. Otherwise we can solve the homogeneous problem with solution $\tilde{\bfu}$ and set $$\bfu:=\tilde\bfu+\mathcal E_{{\mathcal{O}}}\bfu_\partial-\Bog_{\mathcal{O}}(\Div \mathcal E_{{\mathcal{O}}}\bfu_\partial)$$ where $$\mathcal E_{{\mathcal{O}}}:W^{s-1/p,p}(\partial{\mathcal{O}})\rightarrow W^{s,p}({\mathcal{O}})$$
is a continuous linear extension operator and $\Bog_{\mathcal{O}}$ the Bogovskii-operator. The latter solves the divergence equation (with respect to homogeneous boundary conditions on $\partial{\mathcal{O}}$) and satisfies 
\begin{align}\label{eq:bog}
\Bog_{\mathcal{O}}\Div :W^{s,p}\cap \bigg\{\bfw:\,\int_{\partial{\mathcal{O}}}\bfw\cdot\bfn\,\dd\mathcal H^{n-1}=0\bigg\}\rightarrow W^{s,p}_0({\mathcal{O}})
\end{align} for all $s\geq 1$ and $p\in (1,\infty)$. See \cite{Ga}[Section III.3] for the case $s\in\N_0$, the case of fractional $s$ follows by interpolation. 

Our assumption \small$n\big(\frac{1}{p}-\frac{1}{2}\big)+1\leq  s$\normalsize\, implies $W^{s-2,p}({\mathcal{O}})\hookrightarrow W^{-1,2}({\mathcal{O}})$ and $W^{s-1/p,p}(\partial{\mathcal{O}})\hookrightarrow W^{1/2,2}(\partial{\mathcal{O}})$ such that a unique weak solution $(\bfu,\pi)\in W^{1,2}_{0,\Div}({\mathcal{O}})\times L^2_{\perp}({\mathcal{O}})$ to \eqref{eq:Stokes} exists. Furthermore, let us suppose that $\bfu$ and $\pi$ are sufficiently smooth. We will remove this restriction at the end of the proof. 
By assumption there is $\ell\in\mathbb N$ and functions $\varphi_1,\dots,\varphi_\ell\in\bfB^{\theta}_{\rho,p}(\mathbb R^{n-1})$ satisfying \ref{A1}--\ref{A3}.
We clearly find an open set $\mathcal U^0\Subset{\mathcal{O}}$ such that ${\mathcal{O}}\subset \cup_{j=0}^\ell \mathcal U^j$. Finally, we consider a decomposition of unity $(\xi_j)_{j=0}^\ell$ with respect to the covering
$\mathcal U^0,\dots,\mathcal U^\ell$ of ${\mathcal{O}}$. 
For $j\in\{1,\dots,\ell\}$ we consider the extension $\bfPhi_j$ of $\varphi_j$ given by \eqref{est:ext} with inverse $\bfPsi_j$.
 
 Let us fix $j\in\{1,\dots,\ell\}$ and assume, without loss of generality, that the reference point $y_j=0$ and that the outer normal at~$0$ is pointing in the negative $x_n$-direction (this saves us some notation regarding the translation and rotation of the coordinate system).
We multiply $\bfu$ by $\xi_j$ and obtain for $\bfu_j:=\xi_j\bfu$, $\Pi_j:=\xi_j\pi$ and $\bff_j:=\xi_j\bff$ the equation
\begin{align}\label{eq:Stokes2}
\Delta \bfu_j-\nabla\Pi_j=[\Delta,\xi_j]\bfu-[\nabla,\xi_j]\Pi-\bff_j,\quad\Div\bfu_j=\nabla\xi_j\cdot\bfu,\quad\bfu_j|_{\partial{\mathcal{O}}}=0,
\end{align}
with the commutators $[\Delta,\xi_j]=\Delta\xi_j+2\nabla\xi_j\cdot\nabla$ and $[\nabla,\xi_j]=\nabla\xi_j$.
Finally, we set $\bfv_j:=\bfu_j\circ\bfPhi_j$, $\theta_j:=\Pi_j\circ\bfPhi_j$, $\bfg_j:=\mathrm{det}(\nabla\bfPhi_j)([\Delta,\xi_j]\bfu-[\nabla,\xi_j]\Pi-\bff_j)\circ\bfPhi_j$, $h_j=\mathrm{det}(\nabla\bfPhi_j)(\nabla\xi_j\cdot\bfu)\circ\bfPhi_j$ and 
obtain the equations
\begin{align}\label{eq:Stokes3}
\begin{aligned}
&\Div\big(\bfA_j\nabla\bfv_j)-\Div(\mathbf{B}_j\theta_j)=\bfg_j,
\quad\mathbf{B}_j^\top:\nabla\bfv_j=h_j,\quad\bfv_j|_{\partial\mathbb H}=0,
\end{aligned}
\end{align}
where $\bfA_j:=\mathrm{det}(\nabla\bfPhi_j)\nabla\bfPsi_j^\top\circ\bfPhi_j\nabla\bfPsi_j\circ\bfPhi_j$ and $\mathbf{B}_j:=\mathrm{det}(\nabla\bfPhi_j)\nabla\bfPsi_j\circ\bfPhi_j$
 (note that we have $\Div\mathbf{B}_j=0$ due to the Piola identity).
 This can be rewritten as
\begin{align}\label{eq:Stokes3}
\begin{aligned}
\Delta\bfv_j-\nabla\theta_j&=\Div\big((\mathbb I_{n\times n}-\bfA_j)\nabla\bfv_j)+\Div((\mathbf{B}_j-\mathbb I_{n\times n})\theta_j)+\bfg_j,\\
&\Div\bfv_j=(\mathbb I_{n\times n}-\mathbf{B}_j)^\top:\nabla\bfv_j+h_j,\quad\bfv_j|_{\partial\mathbb H}=0.
\end{aligned}
\end{align}


Setting 
\begin{align*}
\mathcal S (\bfv,\theta)&=\mathcal S _1(\bfv)+\mathcal S_2 (\theta),\\
\mathcal S_1 (\bfv)&=\Div\big((\mathbb I_{n\times n}-\bfA_j)\nabla\bfv),\\
\mathcal S_2 (\theta)&=\Div((\mathbf{B}_j-\mathbb I_{n\times n})\theta),\\
\mathfrak s(\bfv)&=(\mathbb I_{n\times n}-\mathbf{B}_j)^\top:\nabla\bfv,
\end{align*} we can finally write
\eqref{eq:Stokes3} as
\begin{align}\label{eq:Stokes3'}
\begin{aligned}
\Delta\bfv_j-\nabla\theta_j=\mathcal S (\bfv_j,\theta_j)+\bfg_j,
\quad\Div\bfv_j=\mathfrak s(\bfv_j)+h_j,\quad\bfv_j|_{\partial\mathbb H}=0,
\end{aligned}
\end{align}
in $\mathbb H.$
Estimates for the Stokes system on the half space are well-known: We apply \cite[Thm. IV 2.1]{Ga} to \eqref{eq:Stokes3'} which yields
\begin{align}\label{eq:0201a}
\|\nabla\bfv_j\|_{L^p_x}+\|\theta_j\|_{L^p_x}\lesssim \|\mathcal S (\bfv_j,\theta_j)+\bfg_j\|_{W^{-1,p}_x}+\|\mathfrak s(\bfv_j)+h_j\|_{L^p_x}.
\end{align}
Similarly, we obtain from \cite[Thm. IV 3.3]{Ga} for $k\geq2$
\begin{align}\label{eq:0201b}
\|\nabla^k\bfv_j\|_{L^p_x}+\|\nabla^{k-1}\theta_j\|_{L^p_x}\lesssim \|\nabla^{k-2}(\mathcal S (\bfv_j,\theta_j)+\bfg_j)\|_{L^{p}_x}+\|\nabla^{k-1}(\mathfrak s(\bfv_j)+h_j)\|_{L^p_x}.
\end{align}
Since $\bfv_j$ is compactly supported (with support included in $\xi_j\circ\bfPhi_j$) we conclude by Poincar\'e's inequality
\begin{align}\label{eq:0201c}
\|\bfv_j\|_{W^{k,p}_x}+\|\theta_j\|_{W^{k-1,p}_x}\lesssim \|\mathcal S (\bfv_j,\theta_j)+\bfg_j\|_{W^{k-2}_x}+\|\mathfrak s(\bfv_j)+h_j\|_{W^{k-1,p}_x}
\end{align}
for all $k\geq 1$. Interpolation implies
\begin{align}\label{eq:0201c}
\|\bfv_j\|_{W^{s,p}_x}+\|\theta_j\|_{W^{s-1,p}_x}\lesssim \|\mathcal S (\bfv_j,\theta_j)+\bfg_j\|_{W^{s-2}_x}+\|\mathfrak s(\bfv_j)+h_j\|_{W^{s-1,p}_x}
\end{align}
for all $s\geq 1$. Our remaining task consists in estimating the right-hand side.
In order to estimate $\|\mathcal S(\bfv,\theta)\|_{W^{s-2,p}_x}$ and $\|\mathfrak s(\bfv)\|_{W^{s-1,p}_x}$ we
use the Sobolev multiplier norm introduced in \eqref{eq:SoMo}. 
By our assumptions on $\varphi_j$ we infer from
 \eqref{eq:MSa} and \eqref{eq:MSb} that $\phi_j\in \mathcal M(W^{s-1/p,p}(\mathbb H))$. Thus $\bfPhi_j\in \mathcal M(W^{s,p}(\mathbb H))$ by \eqref{eq:Phi},  \eqref{eq:MS}, and \eqref{eq:SMp} and $\bfPsi_j\in \mathcal M(W^{s,p}(\mathbb H))$ by \eqref{eq:SMPhiPsi}. Hence we obtain
by \eqref{J}, \eqref{lem:9.4.1} and the definitions of $\bfA_j$ and $\bfPhi_j$ 
\begin{align*}
\|\mathcal S_1(\bfv)\|_{W^{s-2,p}(\mathbb H)}&\lesssim \sup_{\|\bfw\|_{W^{s-1,p}_x}\leq 1}\|(\mathbb I_{n\times n}-\bfA_j)\bfw\|_{W^{s-1,p}(\mathbb H)}\|\nabla\bfv\|_{W^{s-1,p}_x}\\
&\lesssim \sup_{\|\bfw\|_{W^{s-1,p}(\mathbb H)}\leq 1}\|(1-\mathrm{det}(\nabla\bfPhi_j))\bfw\|_{W^{s-1,p}(\mathbb H)}\|\bfv\|_{W^{s,p}_x}\\
&+ \sup_{\|\bfw\|_{W^{s-1,p}_x}\leq 1}\|\mathrm{det}(\nabla\bfPhi_j)(\mathbb I_{n\times n}-\nabla\bfPsi_j^\top\circ\bfPhi_j)\bfw\|_{W^{s-1,p}(\mathbb H)}\|\bfv\|_{W^{s,p}_x}\\
&+ \sup_{\|\bfw\|_{W^{s-1,p}_x}\leq 1}\|\mathrm{det}(\nabla\bfPhi_j)\nabla\bfPsi_j^\top\circ\bfPhi_j(\mathbb I_{n\times n}-\nabla\bfPsi_j\circ\bfPhi_j)\bfw\|_{W^{s-1,p}(\mathbb H)}\|\bfv\|_{W^{s,p}_x}\\
&\lesssim \|\mathcal T\phi_j\|_{\mathcal M(W^{s,p}(\mathbb H))}\|\bfv\|_{W^{s,p}_x}\\&+ \|\bfPhi_j\|_{\mathcal M(W^{s,p}(\mathbb H))}^n\sup_{\|\bfw\|_{W^{s-1,p}_x}\leq 1}\|(\mathbb I_{n\times n}-\nabla\bfPsi_j\circ\bfPhi_j)\bfw\|_{W^{s-1,p}(\mathbb H)}\|\bfv\|_{W^{s,p}_x}\\
&+ \|\bfPhi_j\|_{\mathcal M(W^{s,p}(\mathbb H))}^n\|\bfPsi_j\|_{\mathcal M(W^{s,p}(\mathbb H))}\sup_{\|\bfw\|_{W^{s-1,p}_x}\leq 1}\|(\mathbb I_{n\times n}-\nabla\bfPsi_j\circ\bfPhi_j)\bfw\|_{W^{s-1,p}(\mathbb H)}\|\bfv\|_{W^{s,p}_x}\\
&\lesssim \Big(\|\mathcal T\phi_j\|_{\mathcal M(W^{s,p}(\mathbb H))}+ \sup_{\|\bfw\|_{W^{s-1,p}_x}\leq 1}\|(\mathbb I_{n\times n}-\nabla\bfPsi_j\circ\bfPhi_j)\bfw\|_{W^{s-1,p}(\mathbb H)}\Big)\|\bfv\|_{W^{s,p}_x},
\end{align*}
where 
\begin{align*}
\sup_{\|\bfw\|_{W^{s-1,p}_x}\leq 1}&\|(\mathbb I_{n\times n}-\nabla\bfPsi_j\circ\bfPhi_j)\bfw\|_{W^{s-1,p}(\mathbb H)}\\&=\sup_{\|\bfw\|_{W^{s-1,p}_x}\leq 1}\|(\mathbb I_{n\times n}-\mathrm{det}(\nabla\bfPsi_j)\mathrm{cof}(\nabla\bfPhi_j\circ\bfPhi_j))\bfw\|_{W^{s-1,p}(\mathbb H)} \\&\leq\sup_{\|\bfw\|_{W^{s-1,p}_x}\leq 1}\|(1-\mathrm{det}(\nabla\bfPsi_j))\bfw\|_{W^{s-1,p}(\mathbb H)}\\
&+\sup_{\|\bfw\|_{W^{s-1,p}_x}\leq 1}\|\mathrm{det}(\nabla\bfPsi_j)(\mathbb I_{n\times n}-\mathrm{cof}(\nabla\bfPhi_j\circ\bfPhi_j))\bfw\|_{W^{s-1,p}(\mathbb H)}\\
&\lesssim \|\mathcal T\phi_j\|_{\mathcal M(W^{s,p}(\mathbb H))}\|\bfPsi_j\|_{\mathcal M(W^{s,p}(\mathbb H))}^n\lesssim \|\mathcal T\phi_j\|_{\mathcal M(W^{s,p}(\mathbb H))}.
\end{align*}
So we finally have
\begin{align*}
\|\mathcal S_1(\bfv)\|_{W^{s-2,p}(\mathbb H)}&\lesssim \|\mathcal T\phi_j\|_{\mathcal M(W^{s,p}(\mathbb H))}\|\bfv\|_{W^{s,p}(\mathbb H)}
\end{align*}
and, similarly, 
\begin{align*}
\|\mathcal S_2(\theta)\|_{W^{s-2,p}(\mathbb H)}&\lesssim \sup_{\|\bfw\|_{W^{s-1,p}(\mathbb H)}\leq 1} \|(\mathbf{B}_j-\mathbb I_{n\times n})\bfw\|_{W^{s-1,p}(\mathbb H)}\|\theta\|_{W^{s-1,p}(\mathbb H)}\\
&\lesssim \|\mathcal T\phi_j\|_{\mathcal M(W^{s,p}(\mathbb H))}\|\theta\|_{W^{s-1,p}(\mathbb H)},
\end{align*}
as well as
\begin{align*}
\|\mathfrak s(\bfv)\|_{W^{s-1,p}(\mathbb H)}
&\lesssim \sup_{\|\bfw\|_{W^{s-1,p}(\mathbb H)}\leq 1}\|(\mathbf{B}_j-\mathbb I_{n\times n})\bfw\|_{W^{s-1,p}(\mathbb H)}\|\nabla\bfv\|_{W^{s-1,p}(\mathbb H)}\\
&\lesssim \|\mathcal T\phi_j\|_{\mathcal M(W^{s,p}(\mathbb H))}\|\bfv\|_{W^{s,p}(\mathbb H)},
\end{align*}
By \eqref{eq:MS} we have
\begin{align}\label{eq:MS'}
\|\mathcal T\phi_j\|_{\mathcal M(W^{s,p}(\mathbb H))}\lesssim \|\varphi_j\|_{\mathcal M(W^{s-1/p,p}(\mathbb H))}.
\end{align}
Finally, in the case $p\leq n$ the right-hand side can be bounded by the Lipschitz constant due to by \eqref{eq:MSa} and the embedding $\bfB_{p,p}^{\theta}\hookrightarrow\bfB^{s-1/p}_{\rho,p}$ for $\theta>s-1/p$ and $\varrho$ satisfying \eqref{eq:SMp}. Hence it is conveniently small by our assumption.
If $p>n$ we have by \eqref{eq:MSb}
 \begin{align}\label{eq:MS''}
 \begin{aligned}
 \|\mathcal T\phi_j\|_{\mathcal M(W^{s,p}(\mathbb H))}&\lesssim \|\nabla(\mathcal T\phi_j)\|_{W^{s-1,p}(\mathbb H)}\lesssim \|\varphi_j\|_{W^{s-1/p,p}(\R^{n-1})}\\&\lesssim \|\varphi_j\|_{W^{\theta,p}(\R^{n-1})}^\alpha  \|\varphi_j\|_{W^{1,p}(\R^{n-1})}^{1-\alpha}\lesssim \|\varphi_j\|_{W^{\theta,p}(\R^{n-1})}^\alpha  \|\varphi_j\|_{W^{1,\infty}(\R^{n-1})}^{1-\alpha}
 \end{aligned}
 \end{align}
 for an appropriate choice of $\alpha\in(0,1)$. This is again suitably small. We conclude that
 \begin{align}\label{eq:0301}
 \|\mathcal S(\bfv_j,\theta_j)\|_{W^{s-2,p}_x}+\|\mathfrak s(\bfv)\|_{W^{s-1,p}_x}\leq \delta\big(\|\bfv_j\|_{W^{s,p}(\mathbb H)}+\|\theta_j\|_{W^{s-1,p}(\mathbb H)}\big)
 \end{align}
 for some small $\delta>0$. On the other hand, we have
 \begin{align*}
 \|\bfg_j\|_{W^{s-2,p}_x}&\lesssim \|\bfu\circ\bfPhi_j\|_{W^{s-1,p}_x}+ \|\pi\circ\bfPhi_j\|_{W^{s-2,p}_x}+ \|\bff\circ\bfPhi_j\|_{W^{s-2}_x}\\
& \lesssim \|\bfu\|_{W^{s-1,p}_x}+ \|\pi\|_{W^{s-2,p}_x}+\|\bff\|_{W^{s-2}_x},
 \end{align*}
 where the hidden constant depends on $\mathrm{det}(\nabla\bfPhi_j)$ and  $\|\bfPhi_j\|_{\mathcal M(W^{s,p}(\mathbb H))}$ being controlled by \eqref{eq:detJ},
 \eqref{eq:MS'} and \eqref{eq:MS''} (see \eqref{lem:9.4.1} for the composition with Sobolev multipliers).
 Similarly, we obtain
 \begin{align*}
 \|h_j\|_{W^{s-1,p}_x}&\lesssim \|\bfu\|_{W^{s-1,p}_x}.
 \end{align*}

 Plugging this and \eqref{eq:0301} into \eqref{eq:0201c} shows for all $j\in\{1,\dots,\ell\}$
 \begin{align}\label{almost0}
\|\bfv_j\|_{W^{s,p}_x}+\|\theta_j\|_{W^{s-1,p}_x}\lesssim
 \|\bfu\|_{W^{s-1}_x}+ \|\pi\|_{W^{s-2}_x}+\|\bff\|_{W^{s-2}_x}
\end{align}
provided $\delta$ is sufficiently small. Clearly, the same estimate (even without the first two terms on the right-hand side) holds for $j=0$ by local regularity theory for the Stokes system.
Choosing $s_0\in\R$ such that $W^{1,2}({\mathcal{O}})\hookrightarrow W^{s_0,p}({\mathcal{O}})$, there is $\alpha\in(0,1)$ such that
 \begin{align*}
 \|\bfu\|_{W^{s-1,p}_x}&\leq \|\bfu\|_{W^{s,p}_x}^{\alpha}\|\bfu\|_{W^{s_0,p}_x}^{1-\alpha}\lesssim \|\bfu\|_{W^{s,p}_x}^{\alpha}\|\bfu\|_{W^{1,2}_x}^{1-\alpha}\lesssim\|\bfu\|_{W^{s,p}_x}^{\alpha}\|\bff\|_{W^{-1,2}_x}^{1-\alpha}\lesssim\|\bfu\|_{W^{s,p}_x}^{\alpha}\|\bff\|_{W^{s-2,p}_x}^{1-\alpha}
 \end{align*}
 by the assumption \small$n\big(\frac{1}{p}-\frac{1}{2}\big)+1\leq  s$\normalsize\, and the standard energy estimate for the Stokes system. Hence we obtain
  \begin{align}\label{almost1}
 \|\bfu\|_{W^{s-1,p}_x}&\leq\kappa\|\bfu\|_{W^{s,p}_x}+c(\kappa)\|\bff\|_{W^{s-2,p}_x}
 \end{align}
for any $\kappa>0$. Similarly,
  \begin{align}\label{almost2}
 \|\pi\|_{W^{s-2,p}_x}&\leq\kappa\|\pi\|_{W^{s-1,p}_x}+c(\kappa)\|\bff\|_{W^{s-2,p}_x}
 \end{align}
 using that $\|\pi\|_{L^2_x}\lesssim \|\bff\|_{W^{-1,2}_x}$ as well.
 Plugging \eqref{almost1} and \eqref{almost2} into \eqref{almost0} summing over $j=0,1,\dots,\ell$ and choosing
 $\kappa$ small enough proves the claim provided $\bfu$ and $\pi$ are sufficiently smooth. Let us finally remove this assumption which is not a priori given.
Applying a standard regularisation procedure (by convolution with mollifying kernel) to the functions $\varphi_1,\dots,\varphi_\ell$ from \ref{A1}--\ref{A3} in the parametrisation of $\partial{\mathcal{O}}$ we obtain a smooth boundary. Classically, the solution to the corresponding Stokes system is smooth.
 Such a procedure is standard and has been applied, for instance, in \cite[Section 4]{CiMa}. It is possible to do this in a way that the original domain is included in the regularised domain to which we extend the function $\bff$ by means of an extension operator.
 The regularisation applied to the $\varphi_j's$ converges on all Besov spaces with $p<\infty$. It does not converge on $W^{1,\infty}(\R^{n-1})$, but the regularisation does not expand the $W^{1,\infty}(\R^{n-1})$-norm, which is sufficient. Following the arguments above we obtain
 \eqref{eq:main} for the regularised problem with a uniform constant. The limit passage is straightforward since \eqref{eq:Stokes} is linear.
\end{proof}

\begin{remark}\label{rem:stokes}
In Section \ref{sec:accest} we have to apply Theorem \ref{thm:stokessteady} in the case $n=2$ to the domain $\mathcal O={\Omega}_{\eta(t)}$ for a fixed $t$. We exclude self-intersection and degeneracy by assumption.
 In the framework of Theorem \ref{thm:main}
we have $\eta\in L^\infty(I;W^{2,2}(\omega))$ and ${\Omega}_{\eta(t)}$ is defined in accordance with \eqref{eq:2612}. We must argue that $\partial{\Omega}_\eta\in \bfB^{s}_{2,2}({\mathcal{O}})$ (in the sense of Definition \ref{def:besovboundary}) for some $s>\frac{3}{2}$ and has a small local Lipschitz constant (both uniformly in time). While the Besov regularity is initially clear, we have to introduce local coordinates to control the Lipschitz constant appropriately. Eventually, we must check the Besov regularity again.
Given $x_0\in \partial{\Omega}_{\eta(t)}$ for some $t\in I$ fixed we can rotate the coordinate system such that $\bfn_{\eta(t)}(y(x_0))=(0,1)^\top$ (recall that $\bfn_{\eta(t)}$ is well-defined since $\partial_y\bfvarphi_\eta\neq 0$ by assumption). Accordingly, it holds
\begin{align*}\partial_y\bfvarphi_{\eta(t)}(y(x_0))=\begin{pmatrix}\partial_y\varphi^1_{\eta(t)}(y(x_0))\\\partial_y\varphi^2_{\eta(t)}(y(x_0))\end{pmatrix}=\begin{pmatrix}1\\0\end{pmatrix}.
\end{align*}
Hence the function $\varphi^1_{\eta(t)}$ is invertible in a neighborhood $\mathcal U$ of $y(x_0)$. We define in
$\varphi^1_{\eta(t)}(\mathcal U)$ the function
\begin{align*}
\widetilde\bfvarphi_{x_0}(z)=\begin{pmatrix}z\\ \widetilde\varphi_{x_0}(z))\end{pmatrix}=\begin{pmatrix}z\\ \varphi^2_{\eta(t)}((\varphi^1_{\eta(t)})^{-1}(z))\end{pmatrix}.
\end{align*}
It describes the boundary $\partial{\Omega}_{\eta(t)}$ close to $x_0$.
One easily checks with $z_0=\varphi^1_{\eta(t)}(y(x_0))$ that $\partial_z\widetilde\varphi_{x_0}(z_0)=0$ such that $\partial_z\widetilde\varphi$ is small close to $z_0$. Also, we obtain from the chain rule and the one-dimensional Sobolev embedding that
$\widetilde\varphi_{x_0}\in W^{2,2}$ and hence $\widetilde\varphi_{x_0}\in W^{s,2}=\bfB^s_{2,2}$ for all $s\in(1,2)$ in a neighborhood of $z_0$.
\end{remark}

\begin{remark}
A result in the spirit of Theorem \ref{thm:stokessteady} is proved in \cite[Lemma 3.1]{CS}. However, it only applies in spaces of high regularity with $s\geq 3$ (and only the case $p=2$ is considered) which is too restrictive for our application in Section \ref{sec:reg}.
Moreover, it is assumed that the global parametrisation is a small perturbation of a smooth reference domain (such as the half space). 
First of all, the assumption of a global parametrisation restricts the result to applications ins fluid-structure interaction as generally only local charts are available on bounded domains. Second, global smallness is a small data assumption, while local smallness (as the small Lipschitz constant) can be achieved by local re-parametrisation as in Remark \ref{rem:stokes}.
\end{remark}

\section{Local strong solutions}
\label{sec:strong}
In this subsection we prove the existence of a unique strong solution to \eqref{1}--\eqref{2} which exists locally in time:
\begin{proposition}\label{prop:local}
Suppose that the assumptions of Theorem \ref{thm:main} hold. There is $T^\ast>0$ such that there is a unique strong solution to \eqref{1}--\eqref{2} in $I^\ast=(0,T^\ast)$ in the sense of Definition \ref{def:strongSolution}.
\end{proposition}
 The strategy to prove Proposition \ref{prop:local} is rather standard and similar to previous papers \cite{Le,GraHil,GraHilLe}:
\begin{itemize}
\item We transform the system to the reference domain, cf. Lemma \ref{lem1}.
\item We linearise the system from Lemma \ref{lem1} and obtain estimates for the linearised system, cf. Lemma \ref{lem2}.
\item We construct a contraction map for the linearised problem in Lemma \ref{lem3} (by choosing the end-time small enough) which gives the local solution to \eqref{1}--\eqref{2}.   
\end{itemize}

\subsection{The transformed problem}
For a solution $(\eta,\bfu,\pi)$ to \eqref{1}--\eqref{2} we define $\overline \pi=\pi\circ\bfPsi_\eta$ and $\overline{\bfu}=\bfu\circ\bfPsi_\eta$, where $\bfPsi_\eta$ is defined in \eqref{map}. We also introduce
\begin{align*}
\bfh_\eta(\overline\bfu)&=-(J_{\eta}-J_{\eta_0})\partial_t\overline\bfu- J_\eta(\nabla \bfPsi_\eta^{-1}\circ\bfPsi_\eta\nabla\overline\bfu)\big(\partial_t\bfPsi_\eta^{-1}\circ\bfPsi_\eta+\overline\bfu\big)+J_\eta\bff\circ \bfPsi_\eta^{-1},\\
\bfA_\eta&=J_\eta\big(\nabla \bfPsi_\eta^{-1}\circ\bfPsi_\eta\big)^{\top}\nabla \bfPsi_\eta^{-1}\circ\bfPsi_\eta,\quad\mathbf{B}_\eta=J_\eta\nabla \bfPsi_\eta^{-1}\circ\bfPsi_\eta,\\
\bfH_\eta(\overline\bfu,\overline\pi)&=(\bfA_{\eta_0}-\bfA_\eta)\nabla\overline\bfu-(\mathbf{B}_{\eta_0}-\mathbf{B}_{\eta})\overline\pi,\quad h_\eta(\overline\bfu)=(\mathbf{B}_{\eta_0}-\mathbf{B}_{\eta}):\nabla\overline\bfu,
\end{align*}
where $J_\eta=|\mathrm{det}\nabla\bfPsi_\eta|$.
We see that $(\eta,\overline\bfu,\overline\pi)$ is a strong solution to the coupled system
\begin{align}\label{momref}
J_{\eta_0}\partial_t\overline\bfu+\Div\big(\mathbf{B}_{\eta_0}\overline\pi\big)-\Div\big(\bfA_{\eta_0}\nabla\overline\bfu\big)&=\bfh_\eta(\overline\bfu)-\Div\bfH_\eta(\overline\bfu,\overline\pi),\\
\label{divref}\mathbf{B}_{\eta_0}:\nabla\overline\bfu&=h_\eta(\overline\bfu),\\
\label{beamref}
 \partial_t^2\eta-\partial_t\partial_y^2\eta+\partial_y^4\eta&=-\bfn \big(\bfA_{\eta_0}\nabla\overline\bfu-\mathbf{B}_{\eta_0}\overline\pi\big)\circ\bfvarphi\,\bfn\\
&+\bfn \bfH_\eta(\overline\bfu,\overline\pi)\circ\bfvarphi\,\bfn,
\nonumber
\end{align}
in $I\times\Omega$ and we have
\begin{align}\label{boundaryref'}
\overline\bfu\circ {\bfvarphi} =\partial_t \eta \bfn\quad \text{on}\quad I\times\omega.
\end{align}
 Here equations \eqref{momref} and \eqref{beamref} are understood in the strong sense (satisfied a.a. in $I\times\Omega$ and $I\times\omega$ respectively).

We call $(\eta,\overline\bfu,\overline\pi)$ a strong solution to \eqref{momref}--\eqref{boundaryref'} provided
\begin{align}\label{datasetref}
\begin{aligned}
\eta \in W^{1,\infty} \big(I; W^{1,2}(\omega) \big)\cap W^{1,2} \big(I; W^{2,2}(\omega) \big)\cap  L^\infty \big(I; W^{3,2}(\omega) \big) \cap W^{2,2}(I;L^2(\omega)),\\
 \overline\bfu \in L^\infty \big(I; W^{1,2}(\Omega) \big)\cap  L^2 \big(I; W^{2,2}(\Omega) \big)\cap  W^{1,2} \big(I; L^{2}(\Omega) \big),\quad\overline\pi\in  L^2 \big(I; W^{1,2}(\Omega) \big).
 \end{aligned}
\end{align}
Note that we construct a weak solution to equation \eqref{beamref} meaning we have
\begin{align}\label{beamrefweek}
\begin{aligned}
\int_I \int_\omega \big(\partial_t \eta\, \partial_t\phi&-\partial_t\partial_y\eta\,\partial_y \phi-
 g\, \phi \big)\dy\dt-\int_I\int_\omega \partial_y^2\eta\,\partial_y^2 \phi\dy\dt\\&=-\int_\omega\bfn \big(\bfA_{\eta_0}\nabla\overline\bfu\big)-\mathbf{B}_{\eta_0}\overline\pi\big)\circ\bfvarphi\,\bfn\phi\dy\dt\\
 &+\int_\omega\bfn \bfH_\eta(\overline\bfu,\overline\pi)\circ\bfvarphi\,\bfn\phi\dy\dt
 \end{aligned}
\end{align}
for all $\phi\in C^\infty(\overline I\times\omega)$. However, one can use the regularity properties \eqref{datasetref} to infer
that $\eta\in L^2(I;W^{4,2}(\omega))$ such that all quantities in
\eqref{beamrefweek} are, in fact, $L^2$-functions and we have indeed a strong solution.

We obtain the following characterisation regarding \eqref{momref}--\eqref{boundaryref}.
\begin{lemma}\label{lem1}
Suppose that the dataset $(\bff, g, \eta_0,  \bfu_0, \eta_1)$
satisfies \eqref{dataset} and \eqref{dataset'}. Then $(\eta,\bfu,\pi)$ is a strong solution to \eqref{1}--\eqref{2} (in the sense of Definition \ref{def:strongSolution})
if and only if  $(\eta,\overline\bfu,\overline\pi)$ is a strong solution to \eqref{momref}--\eqref{boundaryref'}.
\end{lemma}
\begin{proof}
Transforming the momentum equation to the reference domain we obtain
\begin{align*}
J_{\eta}\partial_t\overline\bfu+\Div\big(\mathbf{B}_{\eta}\overline\pi\big)-\Div\big(\bfA_{\eta}\nabla\overline\bfu\big)&=- J_\eta(\nabla \bfPsi_\eta^{-1}\circ\bfPsi_\eta\nabla\overline\bfu)\big(\partial_t\bfPsi_\eta^{-1}\circ\bfPsi_\eta+\overline\bfu\big)+\bff\circ \bfPsi_\eta^{-1},\end{align*}
while the incompressibility constraint gives 
$\mathbf{B}_{\eta}^\top:\nabla\overline\bfu=0$. Reordering terms and recalling the definitions
of $\bfB_{\eta}$, $\bfA_{\eta_0}$, $\bff_\eta$ and $\bfH_\eta$ yields
\begin{align}\label{eq:2811}
J_{\eta_0}\partial_t\overline\bfu+\Div\big(\mathbf{B}_{\eta_0}\overline\pi\big)-\Div\big(\bfA_{\eta_0}\nabla\overline\bfu\big)&=\bfh_\eta(\overline\bfu)-\Div\bfH_\eta(\overline\bfu,\overline\pi)\end{align}
and 
\begin{align*}
\bfB_{\eta_0}:\nabla\overline\bfu&=h_\eta(\overline\bfu).
\end{align*}
 Allowing now a couple of test-functions
$(\phi, {\bfphi}) \in C^\infty(\overline{I}\times\omega) \times C^\infty(\overline{I}\times \R^3)$ and $\bfphi\circ\bfvarphi_\eta= \phi{\bfn}$, see Definition \ref{def:weakSolution} (c), rewriting the terms for the fluid equation as above (that is, setting $\overline{\bfphi}=\bfphi\circ\bfPsi_\eta$) and integrating by parts yields
\begin{align*}
\int_I \int_\omega \big(\partial_t \eta\, \partial_t\phi&-\partial_t\partial_y\eta\,\partial_y \phi-
 g\, \phi \big)\dy\dt-\int_I\int_\omega \partial_y^2\eta\,\partial_y^2 \phi\dy\dt\\&=-\int_\omega\bfn \big(\bfA_{\eta_0}\nabla\overline\bfu\big)-\mathbf{B}_{\eta_0}\overline\pi\big)\circ\bfvarphi\,\bfn\phi\dy\dt\\
 &+\int_\omega\bfn \bfH_\eta(\overline\bfu,\overline\pi)\circ\bfvarphi\,\bfn\phi\dy\dt
\end{align*}
due to \eqref{eq:2811}. Note that the terms on the right-hand side are the boundary terms which arise due to the integration by parts.
 This finishes the proof as all the manipulations can be reversed for strong solutions given the regularity of $\overline\bfu$ and $\overline\pi$ assumed in \eqref{datasetref}.
\end{proof}

\subsection{The linearised problem}
We will now consider solutions to the linearised problem for a given right-hand side, that is we analyse for given functions $\bfh,\bfH$ and $h$ the system
\begin{align}\label{momlin}
J_{\eta_0}\partial_t\overline\bfu+\Div\big(\mathbf{B}_{\eta_0}\overline\pi\big)-\Div\big(\bfA_{\eta_0}\nabla\overline\bfu\big)&=\bfh-\Div\bfH,\\
\label{divlin}\mathbf B_{\eta_0}:\nabla\overline\bfu&=h,\\
\nonumber
\partial_t^2\eta-\partial_t\partial_y^2\eta+\partial_y^4\eta&=g-\bfn \big(\bfA_{\eta_0}\nabla\overline\bfu-\mathbf{B}_{\eta_0}\overline\pi\big)\circ\bfvarphi\,\bfn\\
&+\bfn \bfH\circ\bfvarphi\,\bfn,
\label{beamlin}\\\label{boundarylina}
\overline\bfu\circ\bfvarphi& =\partial_t \eta{\bfn}\quad \text{on}\quad I\times\omega,\\
\overline\bfu(0)=\overline\bfu_0,\quad&\eta(0)=\eta_0,\quad\partial_t\eta(0)=\eta_1.\label{boundarylin}
\end{align}
Note that \eqref{momlin}--\eqref{boundarylin} is linear in $(\eta,\overline\bfu,\overline\pi)$ such that we expect strong solutions globally in time belonging to the regularity class specified in \eqref{datasetref}.
\begin{lemma}\label{lem2}
Suppose that the dataset $(\bfh,\bfH,h, g, \eta_0,  \overline{\bfu}_0, \eta_1)$
satisfies
\begin{align}
\nonumber
&\bfh \in L^2\big(I; L^2(\Omega)\big),\quad \bfH \in L^2\big(I; W^{1,2}(\Omega)\big),\quad h\in L^2\big(I; W^{1,2}(\Omega)\big)\cap  W^{1,2}\big(I; W^{-1,2}(\Omega)\big)\cap\{h(0,\cdot)=0\},\\ 
&g \in L^2\big(I; W^{1,2}(\omega)\big), \quad
\eta_0 \in W^{3,2}(\omega) \text{ with } \Vert \eta_0 \Vert_{L^\infty( \omega)} < L,  \quad
\eta_1 \in W^{1,2}(\omega),\label{datasetlin}
\\
& 
\overline\bfu_0\in W^{1,2}(\Omega) \text{ is such that }\overline\bfu_0\circ\bfvarphi =\eta_1 \bfn\text{ and }\mathbf B_{\eta_0}:\nabla\overline\bfu_0=0.\nonumber
\end{align}
Then there is a strong solution to \eqref{momlin}--\eqref{boundarylin} satisfying the estimate
\begin{align}\label{est:lin}
\begin{aligned}
\sup_I\int_\Omega&|\nabla\overline\bfu|^2\dx+\int_I\int_\Omega\big(|\nabla^2\overline\bfu|^2+|\partial_t\overline\bfu|^2+|\overline\pi|^2+|\nabla\overline\pi|^2\big)\dxt\\
 &+\sup_I\int_{\omega}\big(
|\partial_t\partial_y\eta |^2+
|\partial_y^3\eta |^2\big)\dy+\int_I\int_{\omega}
\big(|\partial_t\partial_y^2\eta |^2+|\partial_t^2\eta|^2\big)  \dy\dt\\
&\lesssim \int_\Omega|\nabla\overline\bfu_0|^2\dx+\int_I\int_\Omega\big(|\bfh|^2+|\nabla\bfH|^2+|\nabla h|^2\big)\dxt+\int_I\|\partial_t h\|_{W^{-1,2}(\Omega)}^2\dt\\
&+\int_\omega\big(|\partial_y^3\eta_0|^2+|\partial_y\eta_1|^2\big)\dy+\int_I\int_\omega|\partial_y g|^2\dy\dt\\
&+\overline{\mathcal{E}}(0)+\int_I\int_\Omega|\bfH|^2\dxt+\int_I\int_\omega|g|^2\dy\dt,
\end{aligned}
\end{align}
where the energy $\overline{\mathcal E}$ is given by
\begin{align*}
\overline{\mathcal{E}}(t)&=\frac{1}{2}\int_{\Omega}
 |\overline\bfu(t)|^2   \dx
 +
 \frac{1}{2}\int_{\omega}
|\partial_t\eta |^2  \dy
+ \frac{1}{2}\int_{\omega}
|\partial_y^2\eta |^2  \dy.
\end{align*} 
\end{lemma}
\begin{proof}
Let us initially suppose that $h=0$, that is we have
$\bfB_{\eta_0}:\nabla\overline\bfu=0$.
We proceed formally; a rigorous proof can be obtained by working with a Galerkin approximation. Testing with $(\overline\bfu,\partial_t\eta)$, using $\bfB_{\eta_0}:\nabla\overline\bfu=0$, ellipticity of $\bfA_{\eta_0}$ (which follows from $\|\eta_0\|_{L^\infty_y}<L$) as well as the cancellation of the boundary terms due to \eqref{boundarylina}
 yields
\begin{align*}
\overline{\mathcal{E}}(t)
&+\int_0^t
 \int_{\Omega}|\nabla \overline\bfu |^2 \dx\ds+\int_0^t\int_{\omega}|\partial_t\partial_y\eta |^2 \dy\ds
\\
& \leq
 \overline{\mathcal{E}}(0)
 + \int_0^t\int_{\Omega}\overline\bfu\cdot\bfh\dxs + \int_0^t\int_{\Omega}\nabla\overline\bfu:\bfH\dxs
+
\int_0^t\int_\omega g\,\partial_t\eta\dy\ds.
\end{align*}
This implies
\begin{align}\label{apriorilin}
\begin{aligned}
\sup_I&\int_\Omega|\overline\bfu|^2\dx+\int_I\int_\Omega|\nabla\overline\bfu|^2\dx\dt\\&+\sup_I\int_\omega|\partial_t\eta|^2\dy+\sup_I\int_\omega|\partial_y^2\eta|^2\dy+\int_I\int_\omega|\partial_t\partial_y\eta|^2\dy\dt\\
&\lesssim  \overline{\mathcal{E}}(0)+\int_I\int_\Omega\big(|\bfh|^2+|\bfH|^2\big)\dxt+\int_I\int_\omega|g|^2\dy\dt.
\end{aligned}
\end{align}
Similarly, we can test by $(\partial_t\overline\bfu,\partial_t^2\eta)$ noticing that the coefficients in \eqref{momlin} and \eqref{divlin}
are independent of time
and that $\partial_t\overline\bfu$ and $\partial_t^2\eta$ match again at the boundary due to \eqref{boundarylina}. 
We obtain
\begin{align*}
\int_0^t\int_\Omega|\partial_t\overline\bfu|^2\dxs
&+
 \frac{1}{2}\int_{\Omega}|\nabla \overline\bfu|^2 \dx+\frac{1}{2}\int_{\omega}|\partial_t\partial_y\eta |^2 \dy+\int_0^t\int_{\omega}|\partial_t^2\eta |^2 \dy\ds
\\
& =
 \frac{1}{2}\int_{\Omega}|\nabla \overline\bfu_0|^2 \dx
 + \int_0^t\int_{\Omega}\partial_t\overline\bfu\cdot\bfh\dxs -\int_0^t\int_{\Omega}\partial_t\overline\bfu:\Div\bfH\dxs
\\&+\int_\omega\bfn \bfH\circ\bfvarphi\,\bfn\partial_t^2\eta\dy\dt+\frac{1}{2}\int_{\omega}|\partial_y\eta_1|^2 \dy+
\int_0^t\int_\omega g\,\partial_t^2\eta\dy\ds-\int_0^t\int_\omega \partial_y^4\eta\,\partial_t^2\eta\dy\ds,
\end{align*}
such that, using
\begin{align*}
-\int_{0}^t\int_{\omega}\partial_y^4\eta\,\partial_t^2\eta\dy\ds&=\int_{0}^t\int_{\omega}\partial_t(\partial_y^3\eta\,\partial_t\partial_y\eta)\dy\ds+\int_{0}^t\int_\omega|\partial_t\partial_y^2\eta|^2\dy\ds\\
&\leq\tfrac{1}{2}\sup_{I}\int_\omega|\partial_t\partial_y\eta|^2\dy+2\sup_{I}\int_\omega|\partial_y^3\eta|^2\dy+\int_{I}\int_\omega|\partial_t\partial_y^2\eta|^2\dy\ds,
\end{align*}
we have
\begin{align}
\nonumber
\int_I\int_\Omega|\partial_t\overline\bfu|^2\dxt
&+\sup_I
 \int_{\Omega}|\nabla \overline\bfu|^2 \dx+\sup_I\int_{\omega}|\partial_t\partial_y\eta|^2 \dy+\int_I\int_{\omega}|\partial_t^2\eta |^2\dy\dt\\
&\lesssim   \int_{\Omega}|\nabla \overline\bfu_0|^2 \dx+\int_{\omega}|\partial_y\eta_1|^2 \dy+\int_I\int_\Omega\big(|\bfh|^2+|\nabla\bfH|^2\big)\dxt+\int_I\int_{\partial\Omega}|\bfH|^2\,\dd\mathcal H^1\dt\nonumber\\
&+\int_I\int_\omega|g|^2\dy\dt+\sup_{I}\int_\omega|\partial_y^3\eta|^2\dy+\int_{I}\int_\omega|\partial_t\partial_y^2\eta|^2\dy.\label{apriorilin2}
\end{align}
Now we differentiate the structure equation in space by testing with $\partial_t\partial_y^2\eta$ which leads to
\begin{align*}
 \frac{1}{2}\int_{\omega}
|\partial_t\partial_y\eta |^2  \dy
&+ \frac{1}{2}\int_{\omega}
|\partial_y^3\eta |^2  \dy+\int_0^t\int_{\omega}|\partial_t\partial_y^2\eta |^2 \dy\ds\\&= \frac{1}{2}\int_{\omega}
|\partial_y\eta_1 |^2  \dy
+ \frac{1}{2}\int_{\omega}
|\partial_y^3\eta_0 |^2  \dy+\int_0^t\int_\omega g\,\partial_y^2\partial_t\eta\dy\ds
\\&+\int_I\int_\omega\bfn \big(\bfA_{\eta_0}\nabla\overline\bfu\big)-\mathbf{B}_{\eta_0}\overline\pi\big)\circ\bfvarphi\,\bfn\,\partial_t\partial_y^2\eta\dy\dt\\
 &-\int_I\int_\omega\bfn \bfH\circ\bfvarphi\,\bfn\,\partial_t\partial_y^2\eta\dy\dt.
\end{align*}
Let us explain how to control the last two integrals in the above. By the trace theorem, smoothness of $\bfA_{\eta_0}$, $\bfB_{\eta_0}$ and $\bfvarphi$ as well as interpolation we have
\begin{align*}
\int_I\int_\omega&\bfn \big(\bfA_{\eta_0}\nabla\overline\bfu\big)-\mathbf{B}_{\eta_0}\overline\pi\big)\circ\bfvarphi\,\bfn\,\partial_t\partial_y^2\eta\dy\dt\\&\leq\int_{I}\|\bfn \big(\bfA_{\eta_0}\nabla\overline\bfu\big)-\mathbf{B}_{\eta_0}\overline\pi\big)\circ\bfvarphi\,\bfn\|_{W^{1/2,2}(\omega)}\|\partial_t\partial_y^2\eta\|_{W^{-1/2,2}(\omega)}\dt\\&\leq\,c\int_{I}\big(\|\nabla\overline\bfu\|_{W^{1/2,2}(\partial\Omega)}+\|\overline\pi\|_{W^{1/2,2}(\partial\Omega)}\big)\|\partial_t\eta\|_{W^{3/2,2}(\omega)}\dt\\
&\leq\,c\int_{I}\big(\|\nabla\overline\bfu\|_{W^{1,2}(\Omega_\eta)}+\|\overline\pi\|_{W^{1,2}(\Omega)}\big)\|\partial_t\eta\|_{W^{1,2}(\omega)}^{1/2}\|\partial_t\eta\|_{W^{2,2}(\omega)}^{1/2}\dt\\
&\leq\,\kappa\int_{I}\big(\|\nabla\overline\bfu\|_{W^{1,2}_x}^2+\|\overline\pi\|_{W^{1,2}_x}^2\big)\dt+\kappa\int_I\|\partial_t\eta\|_{W^{2,2}_y}^2\dt+c(\kappa)\int_I\|\partial_t\eta\|_{W^{1,2}_y}^2\dt
\end{align*}
In order to controll the pressure we write similarly to \eqref{eq:pressure}
\begin{align*}
\overline\pi=\overline\pi_0+c_{\overline\pi},
\end{align*}
where $(\overline\pi_0)_\Omega=0$ and $c_{\overline\pi}$ is a function of time only. The latter satisfies
\begin{align*}
c_{\overline\pi}(t)\int_{\omega}\bfn\mathbf{B}_{\eta_0}\circ\bfvarphi\bfn\dy&=\int_{\omega}\bfn\big(\bfA_{\eta_0}\nabla\overline\bfu-\mathbf{B}_{\eta_0}\overline\pi_0\big)\circ\bfvarphi\bfn\dy\\&-\int_{\omega}\bfn\bfH\circ\bfvarphi\bfn\dy+\int_\omega\partial_t^2\eta\dy-\int_\omega g\dy
\end{align*}
due to equation \eqref{beamlin}. Noticing that $\bfB_{\eta_0}$ is uniformly elliptic we infer from Poincar\'e's inequality
\begin{align*}
\int_{I}\|\overline\pi\|^2_{W^{1,2}_x}\dt&\lesssim \int_{I}\|\nabla\overline\pi\|^2_{L^{2}_x}\dt+\int_{I}c_{\overline\pi}^2\dt\\
&\lesssim \int_{I}\|\nabla\overline\pi\|^2_{L^{2}_x}\dt+\int_{I}\int_\omega|\partial_t^2\eta|^2\dy\dt+\int_{I}\int_\omega|g|^2\dy\dt\\&+\int_{I^\ast}\|\bfH\|_{L^{2}(\partial\Omega)}^2\dt+\int_{I}\|\overline\pi_0\|_{L^{2}(\partial\Omega)}^2\dt+\int_{I}\|\nabla\overline\bfu\|_{L^{2}(\partial\Omega)}^2\dt.
\end{align*}
By the trace theorem and $(\overline\pi_0)=0$
we obtain
\begin{align*}
\int_{I}\|\nabla\overline\bfu\|_{W^{1/2,2}(\Omega)}^2\dt
&\lesssim\int_{I^\ast}\|\nabla\overline\bfu\|_{W^{1,2}_x}^{2}\dt,\quad \int_{I}\|\bfH\|_{W^{1/2,2}(\Omega)}^2\dt
\lesssim\int_{I^\ast}\|\bfH\|_{W^{1,2}_x}^{2}\dt,\\
\int_{I}\|\overline\pi_0\|_{L^{2}(\partial\Omega)}^2\dt
&\lesssim\int_{I}\|\overline\pi_0\|_{W^{1,2}_x}^{2}\dt\lesssim \int_{I}\|\nabla\overline\pi_0\|_{L^{2}_x}^{2}\dt=\int_{I}\|\nabla\overline\pi\|_{L^{2}_x}^{2}\dt.
\end{align*}
Similarly, it holds
\begin{align*}
-\int_I\int_\omega\bfn \bfH\circ\bfvarphi\,\bfn\,\partial_t\partial_y^2\eta\dy\dt\leq\int_{I}\|\bfH\|_{W^{1,2}_x}^2\dt+\kappa\int_I\|\partial_t\eta\|_{W^{2,2}_y}^2\dt+c(\kappa)\int_I\|\partial_t\eta\|_{W^{1,2}_y}^2\dt
\end{align*}
We conclude
\begin{align*}
 \sup_I\int_{\omega}
|\partial_t\partial_y\eta |^2  \dy
&+ \sup_I\int_{\omega}
|\partial_y^3\eta |^2  \dy+\int_I\int_{\omega}|\partial_t\partial_y^2\eta |^2 \dy\dt\\&\leq\,c(\kappa)\bigg(\int_{\omega}
|\partial_y\eta_1 |^2  \dy
+ \int_{\omega}
|\partial_y^3\eta_0 |^2  \dy+\int_I\int_\omega |g|^2\dy\dt+\overline{\mathcal E}(0)\bigg)
\\&+\kappa \bigg(\int_{I}\|\nabla\overline\bfu\|_{W^{1,2}_x}^{2}\dt+\int_{I}\|\nabla\overline\pi\|_{L^{2}_x}^{2}\dt\bigg)
\end{align*}
where $\kappa>0$ is arbitrary.
Now we consider the fluid equation \eqref{momlin}--\eqref{divlin} and transform it by means of $\bfPsi_{\eta_0}^{-1}$, that is, we set $\underline \pi=\overline\pi\circ\bfPsi_{\eta_0}^{-1}$ and $\underline{\bfu}=\overline\bfu\circ\bfPsi_{\eta_0}^{-1}$. Arguing as in the beginning of the proof of Lemma \ref{lem1} and noticing that this transformation is independent of time we get
 \begin{align}\label{momlin'}
\partial_t\underline\bfu+\nabla\underline\pi-\Delta\underline\bfu&=J_{\eta_0}^{-1}\big(\bfh\circ\bfPsi^{-1}_{\eta_0}-(\Div\bfH)\circ\bfPsi_{\eta_0}^{-1}\big),\quad
 \Div\underline\bfu=0,
\end{align}
in $I\times\Omega_{\eta_0}$ together with
\begin{align}\label{boundaryref}
\underline\bfu\circ {\bfvarphi_{\eta_0}} =\partial_t \eta {\bfn}\quad \text{on}\quad I\times\omega.
\end{align}
Maximal regularity theory for the classical Stokes problem (and smoothness of $\bfvarphi_{\eta_0}$) yields
\begin{align*}
\int_I\int_{\Omega_{\eta_0}}\big(|\nabla^2\underline\bfu|^2+|\partial_t\underline\bfu|^2+|\nabla\underline\pi|^2\big)\dxt&\lesssim \int_I\int_{\Omega_{\eta_0}}\big(|\bfh\circ\bfPsi_{\eta_0}^{-1}|^2+|(\Div\bfH)\circ\bfPsi_{\eta_0}^{-1}|^2\big)\dxt\\
&+\int_I\|\partial_t\eta\|^2_{W^{3/2,2}(\omega)}\dt
\end{align*}
such that, for $\kappa>0$ arbitrary,
\begin{align*}
\int_I\int_{\Omega}\big(|\nabla^2\overline\bfu|^2+|\partial_t\overline\bfu|^2+|\nabla\overline\pi|^2\big)\dxt&\leq\,c\int_I\int_{\Omega}\big(|\bfh|^2+|\nabla\bfH|^2\big)\dxt+c_\kappa\int_I\int_\omega|\partial_t\partial_y\eta|^2\dy\dt\\
&+\kappa\int_I\int_\omega|\partial_t\partial_y^2\eta|^2\dy\dt
\end{align*}
using again interpolation and transforming back to $\Omega$.

Collecting all the estimate and choosing $\kappa$ small enough proves the claim for $h=0$. Let us now explain how to remove this restriction. We consider the steady Stokes-type system
\begin{align}\label{steadyStokes}
\Div\big(\mathbf B_{\eta_0}\overline p\big)-\Div\big(\bfA_{\eta_0}\nabla\overline\bfv\big)=0,\quad
\bfB_{\eta_0}:\nabla\overline\bfv=h,\quad \overline\bfv|_{\partial\Omega}=0,
\end{align}
in $\Omega$ for a given function $h:\Omega\rightarrow\R$.
We denote the solution operator, which maps $h$ to $\overline \bfv$, by $\mathcal A_{\eta_0}^{-1}$ and 
We claim that the estimates
\begin{align}\label{est:steadystokes}
\int_{\Omega}|\nabla\mathcal A^{-1}_{\eta_0}h|^2\dx\lesssim \int_{\Omega}|h|^2\dx,\quad \int_{\Omega}|\nabla^2\mathcal A^{-1}_{\eta_0}h|^2\dx\lesssim \int_{\Omega}|\nabla h|^2\dx,
\end{align}
hold.
Indeed, transforming \eqref{steadyStokes} by means of $\bfPsi_{\eta_0}^{-1}$ (that is, setting $\underline{p}=\overline p\circ\bfPsi_{\eta_0}^{-1}$ and $\underline{\bfv}=\overline\bfv\circ\bfPsi_{\eta_0}^{-1}$) we obtain the system
\begin{align}\label{steadyStokes0}
\nabla\underline p-\Delta\underline\bfv=0,\quad
\Div\underline\bfv=h\circ \bfPsi_{\eta_0}^{-1},\quad \overline\bfv|_{\partial\Omega_{\eta_0}}=0,
\end{align}
in $\Omega_{\eta_0}$. The estimates
\begin{align}\label{est:steadystoke0s}
 \int_{\Omega}|\nabla\underline\bfv|^2\dx\lesssim \int_{\Omega}|h\circ \bfPsi_{\eta_0}^{-1}|^2\dx,\quad\int_{\Omega}|\nabla^2\underline\bfv|^2\dx\lesssim \int_{\Omega}|\nabla (h\circ \bfPsi_{\eta_0}^{-1})|^2\dx,
\end{align}
are classical and yield \eqref{est:steadystokes} by transformation (and smoothness of $\bfPsi_{\eta_0}$).
If $\overline\bfu$ satisfies \eqref{momlin}--\eqref{divlin}
for a given function $h$, then $\overline\bfu-\mathcal A_{\eta_0}^{-1}h$ satisfies the problem with homogenous constraint (note that \eqref{beamlin} and \eqref{boundarylina} do not change as $\mathcal A_{\eta_0}^{-1}h$ vanishes at the boundary) with the additional term $J_{\eta_0}\partial_t\mathcal A_{\eta_0}^{-1}h$ on the right-hand side of the mometum equation. Applying the previously proved estimate for the problem with homogeneous constraint we obtain the additional term
\begin{align*}
\int_I\int_{\Omega}\big(|\partial_t\mathcal A_{\eta_0}^{-1}h|^2+|\nabla^2\mathcal A_{\eta_0}^{-1}h|^2\big)\dxt
&\lesssim \int_I\int_\Omega|\nabla h|^2\dxt+\int_I\|\partial_t h\|_{W^{-1,2}(\Omega)}^2\dt\\
\sup_I\int_{\Omega}|\nabla\mathcal A_{\eta_0}^{-1}h|^2&\lesssim \int_I\int_{\Omega}\big(|\partial_t\mathcal A_{\eta_0}^{-1}h|^2+|\nabla^2\mathcal A_{\eta_0}^{-1}h|^2\big)\dxt\\
&\lesssim \int_I\int_\Omega|\nabla h|^2\dxt+\int_I\|\partial_t h\|_{W^{-1,2}(\Omega)}^2\dt
\end{align*}
using also \eqref{est:steadystokes} and $h(\cdot,0)=0$. The proof is now completed.
\end{proof}

\subsection{The fixed point argument}
We consider now for $(\zeta,\overline\bfw,\overline q)$ given the problem
\begin{align}\label{momfix}
J_{\eta_0}\partial_t\overline\bfu+\Div\big(\mathbf{B}_{\eta_0}\overline\pi\big)-\Div\big(\bfA_{\eta_0}\nabla\overline\bfu\big)&=\bfh_{\zeta}(\overline\bfw)-\Div\bfH_{\zeta}(\overline\bfw,\overline q),\\
\label{divfix}\bfB_{\eta_0}:\nabla\overline\bfu&=h_{\zeta}(\overline\bfw),\\
\label{beamfix}
 \partial_t^2\eta-\partial_t\partial_y^2\eta+\partial_y^4\eta&=-\bfn \big(\bfA_{\eta_0}\nabla\overline\bfu-\mathbf{B}_{\eta_0}\overline\pi\big)\circ\bfvarphi\,\bfn\\
&+\bfn \bfH_{\zeta}(\overline\bfw,\overline q)\circ\bfvarphi\,\bfn,
\nonumber\\\label{boundaryfix}
\overline\bfu\circ\bfvarphi& =\partial_t \eta {\bfn}\quad \text{on}\quad I\times\omega.
\end{align}
We consider the solution map $\mathscr T_{\eta_0}$ which maps $(\zeta,\overline\bfw,\overline q)$ to the solution $(\eta,\overline\bfu,\overline \pi)$ of \eqref{momfix}--\eqref{boundaryfix} (existence of which follows from Lemma \ref{lem2}). Setting $I^\ast=(0,T^\ast)$ for some small $T^\ast>0$ we must prove that it is a contraction on the space \begin{align*}\mathscr Y^\ast:=
 W^{1,\infty} &\big(I^\ast; W^{1,2}(\omega) \big)\cap W^{1,2} \big(I^\ast; W^{2,2}(\omega) \big)\cap  L^\infty \big(I^\ast; W^{3,2}(\omega) \big) \cap W^{2,2}(I^\ast;L^2(\omega))\\
 & \times L^\infty \big(I^\ast; W^{1,2}(\Omega) \big)\cap  W^{1,2} \big(I^\ast; L^{2}(\Omega) \big)\cap  L^2 \big(I^\ast; W^{2,2}(\Omega) \big)\times L^2 \big(I^\ast; W^{1,2}(\Omega)\big)
\end{align*} complemented with the norm
\begin{align*}
\|(\eta,\overline\bfu,\overline \pi)\|_{\mathscr Y^\ast}^2&:=\sup_{I^\ast}\int_\Omega|\overline\bfu|^2\dx+\int_{I^\ast}\int_\Omega|\nabla\overline\bfu|^2\dx\dt+\sup_{I^\ast}\int_{\omega}\big(
|\partial_t\eta |^2+
|\partial_y^2\eta |^2\big)\dy+\int_{I^\ast}\int_{\omega}
|\partial_t\partial_y\eta |^2\dy\dt\\
&+\sup_{I^\ast}\int_\Omega|\nabla\overline\bfu|^2\dx+\int_{I^\ast}\int_\Omega\big(|\nabla^2\overline\bfu|^2+|\partial_t\overline\bfu|^2+|\overline\pi|^2+|\nabla\overline\pi|^2\big)\dxt\\
 &+\sup_{I^\ast}\int_{\omega}\big(
|\partial_t\partial_y\eta |^2+
|\partial_y^3\eta |^2\big)\dy+\int_{I^\ast}\int_{\omega}
\big(|\partial_t\partial_y^2\eta |^2+|\partial_t^2\eta|^2\big)  \dy\dt
\end{align*}
given by the energy estimate from Lemma \ref{lem2}. This is the content of the following lemma in which we denote by $B_{R}^{\mathscr Y^\ast}(0)$ the ball in $\mathscr Y^\ast$ with radius $R$ around the origin.
\begin{lemma}\label{lem3}
Suppose that $(\bff,g,\eta_0,\bfu_0,\eta_1)$ satisfies \eqref{dataset} and \eqref{dataset'}. There are $R\gg1$ and $T^{\ast}\ll1$ such that $\mathscr T_{\eta_0}:B_{R}^{\mathscr Y^\ast}(0)\cap\{\eta(0)=\eta_0\}\rightarrow B_{R}^{\mathscr Y^\ast}(0)\cap\{\eta(0)=\eta_0\}$ is a contraction.
\end{lemma}
\begin{proof}
First of all, we choose $R$ sufficiently large compared to the dataset $(\bff,g,\eta_0,\bfu_0,\eta_1)$. We intend to control
the Lipschitz constants of the mappings 
\begin{align}\label{eq3011}
\begin{aligned}
\mathscr Y^\ast\ni(\zeta,\overline\bfw,\overline q)&\mapsto \bfh_\zeta(\overline\bfw)\in L^2(I^\ast;L^2(\Omega)),\\
\mathscr Y^\ast\ni(\zeta,\overline\bfw,\overline q)&\mapsto \bfH_\zeta(\overline\bfw,\overline q)\in L^2(I^\ast;W^{1,2}(\Omega)),\\
\mathscr Y^\ast\ni(\zeta,\overline\bfw,\overline q)&\mapsto h_\zeta(\overline\bfw)\in L^2(I^\ast;W^{1,2}(\Omega))\cap W^{1,2}(I^\ast,W^{-1,2}(\Omega)).
\end{aligned}
\end{align} 
Since all of them map $(\eta_0,0,0)$ to the origin, this will also imply that $\mathscr T_{\eta_0}:B_{R}^{\mathscr Y^\ast}(0)\cap\{\eta(0)=\eta_0\}\rightarrow B_{R}^{\mathscr Y^\ast}(0)$.

As far as $\bfH$ is concerned, we have\footnote{Here and in the remainder of this proof the hidden constants depend on $R$ but are independent of $T^\ast$.}
\begin{align*}
\int_{I^\ast}&\|\bfH_{\zeta_1}(\overline\bfw_1,\overline q_1)-\bfH_{\zeta_2}(\overline\bfw_2,\overline q_2)\|_{W^{1,2}_x}^2\dt\\&\lesssim\int_{I^\ast}\| (\bfA_{\eta_0}-\bfA_{ \zeta_1})(\nabla\overline\bfw_1-\nabla\overline\bfw_2)\|_{W^{1,2}_x}^2\dt+\int_{I^\ast}\| (\bfA_{ \zeta_1}-\bfA_{ \zeta_2})\nabla\overline\bfw_2\|_{W^{1,2}_x}^2\dt\\
&+\int_{I^\ast}\| (\bfB_{ \zeta_1}-\bfB_{\eta_0})(\overline q_1-\overline q_2)\|_{W^{1,2}_x}^2\dt+\int_{I^\ast}\| (\mathbf B_{ \zeta_1}-\mathbf B_{ \zeta_2})\overline q_2\|_{W^{1,2}_x}^2\dt,
\end{align*}
where, by the embeddings $W^{1,2}(I^\ast;L^\infty(\omega))\hookrightarrow L^\infty(I^\ast;L^\infty(\omega))$ and $W^{1,2}(\omega)\hookrightarrow L^{\infty}(\omega)$ as well as \eqref{est:psietazeta}--\eqref{est:psi-1etazeta},
\begin{align*}
\int_{I^\ast}&\| (\bfA_{\eta_0}-\bfA_{ \zeta_1})(\nabla\overline\bfw_1-\nabla\overline\bfw_2)\|_{W^{1,2}_x}^2\dt+\int_{I^\ast}\| (\bfA_{ \zeta_1}-\bfA_{ \zeta_2})\nabla\overline\bfw_2\|_{W^{1,2}_x}^2\dt\\
&\lesssim T^{\ast}\sup_{I^\ast}\|\partial_y^2 \zeta_1-\partial_y^2\eta_0\|_{L^\infty_y}^2\sup_{I^\ast}\|\nabla\overline\bfw_1-\nabla\overline\bfw_2\|^2_{L^2_x}+\sup_{I^\ast}\| \partial_y\zeta_1- \partial_y\eta_0\|_{L^\infty_y}^2\int_{ I^{\ast}}\|\nabla^2\overline\bfw_1-\nabla^2\overline\bfw_2\|^2_{L^2_x}\dt\\
&+ T^{\ast}\sup_{I^\ast}\|\partial_y^2 \zeta_1-\partial_y^2 \zeta_2\|_{L^\infty_y}^2\sup_{I^\ast}\|\nabla\overline\bfw_2\|^2_{L^2_x}+\sup_{I^\ast}\| \partial_y\zeta_1- \partial_y\zeta_2\|_{L^\infty_y}^2\int_{ I^{\ast}}\|\nabla^2\overline\bfw_2\|^2_{L^2_x}\dt\\
&\lesssim T^{\ast}\sup_{I^\ast}\|\nabla\overline\bfw_1-\nabla\overline\bfw_2\|^2_{L^2_x}+T^\ast\int_{ I^{\ast}}\|\nabla^2\overline\bfw_1-\nabla^2\overline\bfw_2\|^2_{L^2_x}\dt\\
&+ T^{\ast}\sup_{I^\ast}\|\partial_y^3 \zeta_1-\partial_y^3 \zeta_2\|_{L^2_y}^2+T^\ast\int_{I^\ast}\|\partial_t \partial_y^2\zeta_1-\partial_t \partial_y^2\zeta_2\|_{L^{2}_y}^2\dt\\
&\lesssim T^\ast\|( \zeta_1,\overline\bfw_1,\overline  q_1)-( \zeta_2,\overline\bfw_2,\overline  q_2)\|_{\mathscr Y^\ast}^2
\end{align*}
as well as
\begin{align*}
\int_{I^\ast}&\| (\mathbf{B}_{ \zeta_1}-\mathbf{B}_{\eta_0})(\overline q_1-\overline q_2)\|_{W^{1,2}_x}^2\dt+\int_{I^\ast}\| (\mathbf{B}_{ \zeta_1}-\mathbf{B}_{ \zeta_2})\overline q_2\|_{W^{1,2}_x}^2\dt\\
&\lesssim\sup_{I^\ast}\|\partial_y^2\zeta_1-\partial_y^2\eta_0\|_{L^4_y}^2\int_{I^\ast}\|\overline q_1-\overline q_2\|^2_{L^4_x}\dt+\sup_{I^\ast}\| \partial_y\zeta_1- \partial_y\eta_0\|_{L^\infty_y}^2\int_{ I^{\ast}}\|\nabla\overline q_1-\nabla\overline q_2\|^2_{L^2_x}\dt\\
&+ \sup_{I^\ast}\|\partial_y^2\zeta_1-\partial_y^2\zeta_2\|_{L^4_y}^2\int_{I^\ast}\|\overline q_2\|^2_{L^4_x}\dt+\sup_{I^\ast}\| \partial_y\zeta_1- \partial_y\zeta_2\|_{L^\infty_y}^2\int_{ I^{\ast}}\|\nabla\overline q_2\|^2_{L^2_x}\dt\\
&\lesssim\sup_{I^\ast}\|\zeta_1-\eta_0\|_{W^{5/2,2}_y}^2\|\overline q_1-\overline q_2\|^2_{W^{1,2}_x}\dt+\sup_{I^\ast}\| \partial_y^2\zeta_1- \partial_y^2\eta_0\|_{L^2_y}^2\int_{ I^{\ast}}\|\nabla\overline q_1-\nabla\overline q_2\|^2_{L^2_x}\dt\\
&+ \sup_{I^\ast}\|\zeta_1-\zeta_2\|_{W^{5/2,2}_y}^2\int_{I^\ast}\|\overline q_2\|^2_{W^{1,2}_x}\dt+\sup_{I^\ast}\| \partial_y^2\zeta_1- \partial_y^2\zeta_2\|_{L^2_y}^2\int_{ I^{\ast}}\|\nabla\overline q_2\|^2_{L^2_x}\dt\\
&\lesssim\sqrt{T^\ast}\int_{I^\ast}\|\overline q_1-\overline q_2\|^2_{W^{1,2}_x}\dt+T^\ast\int_{ I^{\ast}}\|\nabla\overline q_1-\nabla\overline q_2\|^2_{L^2_x}\dt\\
&+ \sqrt{T^\ast}\bigg(\sup_{I^\ast}\|\partial_y^3\zeta_1-\partial_y^3\zeta_2\|^2_{L^{2}_y}+\int_{I^\ast}\|\partial_t\partial_y^2\zeta_1-\partial_t\partial_y^2\zeta_2\|^2_{L^2_y}\dt\bigg)+T^\ast\int_{I^\ast}\| \partial_t\partial_y^2\zeta_1- \partial_t\partial_y^2\zeta_2\|_{L^2_y}^2\dt\\
&\lesssim \sqrt{T^\ast}\|( \zeta_1,\overline\bfw_1,\overline  q_1)-( \zeta_2,\overline\bfw_2,\overline  q_2)\|_{\mathscr Y^\ast}^2.
\end{align*}
Note that we also used the embedding
\begin{align*}
W^{1,2}(I^\ast;W^{2,2})\cap L^\infty(I^\ast;W^{3,2}(\Omega))\hookrightarrow C^{1/4}(I^\ast;W^{5/2,2}(\omega))\hookrightarrow C^{1/4}(I^\ast;W^{2,4}(\omega)).
\end{align*}
Similarly, it holds
\begin{align*}
\int_{I^\ast}&\|h_{\zeta_1}(\overline\bfw_1)-h_{\zeta_2}(\overline\bfw_2)\|_{W^{1,2}_x}^2\dt\\
&\lesssim\int_{I^\ast}\| (\mathbf{B}_{\zeta_0}-\mathbf{B}_{ \zeta_1}):(\nabla\overline\bfw_1-\nabla\overline\bfw_2)\|_{W^{1,2}_x}^2\dt+\int_{I^\ast}\| (\mathbf{B}_{ \zeta_1}-\mathbf{B}_{ \zeta_2}):\nabla\overline\bfw_2\|_{W^{1,2}_x}^2\dt\\&\lesssim T^{\ast}\sup_{I^\ast}\|\nabla\overline\bfw_1-\nabla\overline\bfw_2\|^2_{L^2_x}+T^\ast\int_{ I^{\ast}}\|\nabla^2\overline\bfw_1-\nabla^2\overline\bfw_2\|^2_{L^2_x}\dt\\
&+ T^{\ast}\sup_{I^\ast}\|\partial_y^3 \zeta_1-\partial_y^3 \zeta_2\|_{L^2_y}^2+T^\ast\int_{I^\ast}\|\partial_t \partial_y^2\zeta_1-\partial_t \partial_y^2\zeta_2\|_{L^{2}_y}^2\dt\\
&\lesssim T^\ast\|( \zeta_1,\overline\bfw_1,\overline  q_1)-( \zeta_2,\overline\bfw_2,\overline  q_2)\|_{\mathscr Y^\ast}^2.
\end{align*}
As far as the $W^{1,2}_tW^{-1,2}_x$-norm is concerned, we use the embeddings
\begin{align*}
L^2(I^\ast;W^{3,2}(\omega))\cap W^{2,2}(I^\ast;L^2(\omega))\hookrightarrow W^{1,2}(I^\ast;W^{1,4}(\omega)),\\
W^{1,2}(I^\ast;W^{2,2}(\omega))\cap L^\infty(I^\ast;W^{3,2}(\omega))\hookrightarrow C^{1/8}(I^\ast;W^{11/4,2}(\omega)),
\end{align*}
and obtain
\begin{align*}
\int_{I^\ast}&\|\partial_t(h_{\zeta_1}(\overline\bfw_1)-h_{\zeta_2}(\overline\bfw_2))\|_{W^{-1,2}_x}^2\dt\\
&\lesssim\int_{I^\ast}\|\partial_t\mathbf{B}_{ \zeta_1}:(\nabla\overline\bfw_1-\nabla\overline\bfw_2)\|_{W^{-1,2}_x}^2\dt+\int_{I^\ast}\| \partial_t(\mathbf{B}_{ \zeta_1}-\mathbf{B}_{ \zeta_2}):\nabla\overline\bfw_2\|_{W^{-1,2}_x}^2\dt\\
&+\int_{I^\ast}\|(\mathbf{B}_{\eta_0}-\mathbf{B}_{ \zeta_1}):\partial_t(\nabla\overline\bfw_1-\nabla\overline\bfw_2)\|_{W^{-1,2}_x}^2\dt+\int_{I^\ast}\| (\mathbf{B}_{ \zeta_1}-\mathbf{B}_{ \zeta_2}):\nabla\partial_t\overline\bfw_2\|_{W^{-1,2}_x}^2\dt\\
&\lesssim \sup_{I^\ast}\|\nabla\overline\bfw_1-\nabla\overline\bfw_2\|_{L^{2}_x}^2\int_{I^\ast}(1+\|\partial_t\partial_y\zeta_1\|_{L^4_y}^2)\dt+\sup_{I^\ast}\|\nabla\overline\bfw_2\|_{L^{2}_x}^2\int_{I^\ast}\|\partial_t\partial_y(\zeta_1-\zeta_2)\|_{L^4_y}^2\dt\\
&+\sup_{I^\ast}\|\partial_y^2(\zeta_0-\zeta_1)\|_{L^\infty_y}\int_{I^\ast}\|\partial_t(\overline\bfw_1-\overline\bfw_2)\|_{L^{2}_x}^2\dt+\sup_{I^\ast}\|\partial_y^2(\zeta_1-\zeta_2)\|_{L^\infty_y}\int_{I^\ast}\|\partial_t\overline\bfw_2\|_{L^{2}_x}^2\dt\\
&\lesssim \sqrt{T^{\ast}}\sup_{I^\ast}\|\nabla\overline\bfw_1-\nabla\overline\bfw_2\|^2_{L^2_x}+\sqrt{T^\ast}\sup_{I^\ast}\|\partial_y^3(\zeta_1-\zeta_2)\|_{L^2_y}\bigg(\int_{I^\ast}\|\partial_t\partial_y^2(\zeta_1-\zeta_2)\|^2_{L^2_y}\dt\bigg)^{\frac{1}{2}}\\
&+ \sqrt[4]{T^{\ast}}\sup_{I^\ast}\|\partial_t(\overline\bfw_1-\overline\bfw_2)\|_{L^2}^2\dt+\sqrt[4]{T^\ast}\bigg(\sup_{I^\ast}\|\partial_y^3\zeta_1-\partial_y^3\zeta_2\|^2_{L^{2}_y}+\int_{I^\ast}\|\partial_t\partial_y^2\zeta_1-\partial_t\partial_y^2\zeta_2\|^2_{L^2_y}\dt\bigg)\\
&\lesssim \sqrt[4]{T^\ast}\|( \zeta_1,\overline\bfw_1,\overline  q_1)-( \zeta_2,\overline\bfw_2,\overline  q_2)\|_{\mathscr Y^\ast}^2
\end{align*}
Moreover, we gain using again \eqref{est:psi-1eta} and \eqref{est:psi-1etazeta}
\begin{align*}
\int_{I^\ast}&\|\bfh_{\zeta_1}(\overline\bfw_1,\overline q_1)-\bfh_{\zeta_2}(\overline\bfw_2,\overline q_2)\|_{L^{2}_x}^2\dt\\
&\lesssim\int_{I^\ast}\|(J_{ \zeta_1}-J_{\eta_0})(\partial_t\overline\bfw_1-\partial_t\overline\bfw_2)\|_{L^2_x}^2\dt
+\int_{I^\ast}\|(J_{ \zeta_2}-J_{ \zeta_1})\partial_t\overline\bfw_2\|_{L^2_x}^2\dt\\
&+\int_{I^\ast}\|J_{ \zeta_1}\nabla \bfPsi_{ \zeta_1}^{-1}\circ\bfPsi_{ \zeta_1}(\nabla\overline\bfw_1-\nabla\overline\bfw_2)\partial_t\bfPsi_{ \zeta_1}^{-1}\circ\bfPsi_{ \zeta_1}\|_{L^2_x}^2\dt\\
&+\int_{I^\ast}\|\big(J_{ \zeta_1}\nabla \bfPsi_{ \zeta_1}^{-1}\circ\bfPsi_{ \zeta_1}-J_{ \zeta_2}\nabla \bfPsi_{ \zeta_2}^{-1}\circ\bfPsi_{ \zeta_2}\big)\nabla\overline\bfw_2\partial_t\bfPsi_{ \zeta_1}^{-1}\circ\bfPsi_{ \zeta_1}\|_{L^2_x}^2\dt\\
&+\int_{I^\ast}\|J_{ \zeta_2}\nabla \bfPsi_{ \zeta_2}^{-1}\circ\bfPsi_{ \zeta_2}\nabla\overline\bfw_2\big(\partial_t\bfPsi_{ \zeta_1}^{-1}\circ\bfPsi_{ \zeta_1}-\partial_t\bfPsi_{ \zeta_2}^{-1}\circ\bfPsi_{ \zeta_2}\big)\|_{L^2_x}^2\dt\\
&+\int_{I^\ast}\|J_{ \zeta_1}\nabla \bfPsi_{ \zeta_1}^{-1}\circ\bfPsi_{ \zeta_1}(\nabla\overline\bfw_1\overline\bfw_1-\nabla\overline\bfw_2\overline\bfw_2)\|_{L^2_x}^2\dt\\
&+\int_{I^\ast}\|(J_{ \zeta_1}\nabla \bfPsi_{ \zeta_1}^{-1}\circ\bfPsi_{ \zeta_1}-J_{ \zeta_2}\nabla \bfPsi_{ \zeta_2}^{-1}\circ\bfPsi_{ \zeta_2})\nabla\overline\bfw_2\overline\bfw_2\|_{L^2_x}^2\dt\\
&+\int_{I^\ast}\|J_{\zeta_1}(\bff\circ \bfPsi_{ \zeta_1}^{-1}-\bff\circ \bfPsi_{ \zeta_2}^{-1})\|_{L^2_x}^2\dt+\int_{I^\ast}\|(J_{\zeta_1}-J_{\zeta_2})\bff\circ \bfPsi_{ \zeta_2}^{-1}\|_{L^2_x}^2\dt\\
&\lesssim\sup_{I^\ast}\| \partial_y\zeta_1- \partial_y\eta_0\|_{L^\infty_y}^2\int_{I^\ast}\|\partial_t\overline\bfw_1-\partial_t\overline\bfw_2\|_{L^2_x}^2\dt
+\sup_{I^\ast}\| \partial_y\zeta_2- \partial_y\zeta_1\|^2_{L^\infty_y}\int_{I^\ast}\|\partial_t\overline\bfw_2\|_{L^2_x}^2\dt\\
&+T^\ast \sup_{I^\ast}\big(1+\| \partial_y\zeta_1\|^2_{L^\infty_y}\big)\sup_{I^\ast}\big(1+\|\partial_t \zeta_1\|^2_{L^\infty_y}\big)\sup_{I^\ast}\|\nabla\overline\bfw_1-\nabla\overline\bfw_2\|_{L^2_x}^2\\
&+T^\ast \sup_{I^\ast}\| \partial_y\zeta_1- \partial_y\zeta_2\|_{L^\infty_y}^2\sup_{I^\ast}\big(1+\|\partial_t \zeta_1\|_{L^\infty_y}^2\big)\sup_{I^\ast}\|\nabla\overline\bfw_2\|_{L^2_x}^2\\
&+T^\ast \sup_{I^\ast}\big(1+\| \partial_y\zeta_2\|^2_{L^\infty_y}\big)\sup_{I^\ast}\|\partial_t \zeta_1-\partial_t \zeta_2\|_{L^\infty_y}^2\sup_{I^\ast}\|\nabla\overline\bfw_2\|_{L^2_x}^2\\
&+\sup_{I^\ast}\big(1+\| \partial_y\zeta_1\|_{L^\infty_y}^2\big)\sup_{I^\ast}\|\overline\bfw_1-\overline\bfw_2\|^2_{L^4_x}\int_{I^\ast}\|\nabla\overline\bfw_1\|_{L^4_x}^2\dt\\
&+\sup_{I^\ast}\big(1+\| \partial_y\zeta_1\|_{L^\infty_y}^2\big)\sup_{I^\ast}\|\overline\bfw_1\|^2_{L^4_x}\int_{I^\ast}\|\nabla\overline\bfw_1-\nabla\overline\bfw_2\|_{L^4_x}^2\dt\\
&+\sup_{I^\ast}\| \partial_y\zeta_1- \partial_y\zeta_2\|_{L^\infty_y}^2\sup_{I^\ast}\|\overline\bfw_2\|^2_{L^4_x}\int_{I^\ast}\|\nabla\overline\bfw_2\|_{L^4_x}^2\dt\\
&+T^\ast\sup_{I^\ast}\|\nabla\bff\|_{L^2_x}^2\sup_{I^\ast}\| \zeta_1- \zeta_2\|^2_{L^\infty_y}.
\end{align*}
Finally, using the parabolic embeddings
\begin{align*}
L^\infty(I^\ast,L^2(\Omega))\cap L^2(I^\ast;W^{1,2}(\Omega))&\hookrightarrow L^4(I^\ast;L^4(\Omega)),\\
W^{1,2}(I^\ast;L^\infty(\omega))&\hookrightarrow L^\infty(I^\ast;L^\infty(\omega)),
\end{align*}
we obtain
\begin{align*}
\int_{I^\ast}&\|\bfh_{\zeta_1}(\overline\bfw_1,\overline q_1)-\bfh_{\zeta_2}(\overline\bfw_2,\overline q_2)\|_{L^{2}_x}^2\dt\\
&\lesssim T^\ast\|\partial_t\overline\bfw_1-\partial_t\overline\bfw_2\|_{L^2_x}^2\dt
+T^\ast\int_{I^\ast}\|\partial_t( \partial_y^2\zeta_2- \partial_y^2\zeta_1)\|^2_{L^2_y}\dt\\
&+T^\ast \sup_{I^\ast}\|\nabla\overline\bfw_1-\nabla\overline\bfw_2\|_{L^2_x}^2
+T^\ast \sup_{I^\ast}\| \partial_y^2\zeta_1- \partial_y^2\zeta_2\|_{L^\infty_y}^2+T^\ast\sup_{I^\ast}\|\partial_t \zeta_1-\partial_t \zeta_2\|_{L^\infty_y}^2\\
&+\sqrt{T^\ast}\sup_{I^\ast}\|\overline\bfw_1-\overline\bfw_2\|^2_{W^{1,2}_x}+\sqrt{T^\ast}\bigg(\int_{I^\ast}\|\nabla^2\overline\bfw_1-\nabla^2\overline\bfw_2\|_{L^2_x}^2\dt+\sup_{I^\ast}\|\nabla\overline\bfw_1-\nabla\overline\bfw_2\|_{L^2_x}^2\bigg)\\
&+T^\ast\sup_{I^\ast}\| \zeta_1- \zeta_2\|^2_{L^\infty_y}\\
&\lesssim \sqrt{T^\ast}\|( \zeta_1,\overline\bfw_1,\overline  q_1)-( \zeta_1,\overline\bfw_2,\overline  q_2)\|_{\mathscr Y^\ast}^2
\end{align*}
In conclusion, the Lipschitz constants of the mappings in \eqref{eq3011} can be made arbitrarily small if we choose $T^\ast$ appropriately. Combining this observation with the estimate from Lemma \ref{lem2} gives the claim.
\end{proof}

\begin{proof}[Proof of Proposition \ref{prop:local}]
Combining Lemmas \ref{lem1}--\ref{lem3} yields the claim by a standard fixed point argument.
\end{proof}

\section{Regularity estimates}
\label{sec:reg}
This section is devoted to the proof of Theorem \ref{thm:main} for the fluid-structure-interaction problem \eqref{1}--\eqref{2}. With the local strong solution from Proposition \ref{prop:local} at hand we have a sufficiently smooth object such that the following computations are well-defined.
The heart of our analysis is an
acceleration estimate in Proposition \ref{prop2}. It implies that there is no blow-up in finite time such that the global solution can be constructed by gluing local solutions together.

Let us start with
the standard energy estimate
(which is even satisfied by weak solutions).
Let $(\eta,\bfu,\pi)$ be the unique local strong solution from Proposition \ref{prop:local}. We choose
$(\bfu,\partial_t\eta)$ as a test-function in the weak formulation, cf. Definition \ref{def:weakSolution} (c). This yields
\begin{align}\label{eq:energy}
\begin{aligned}
\mathcal{E}(t)
&+\int_0^t
 \int_{\Omega_{\eta}}|\nabla \bfu |^2 \dx\ds+\int_0^t\int_{\omega}|\partial_t\partial_y\eta |^2 \dy\ds
\\
& \leq
 \mathcal{E}(0)
 + \int_0^t\int_{\Omega_\eta}\bfu\cdot\bff\dxs
+
\int_0^{t}\int_\omega g\,\partial_t\eta\dy\ds,\\
\mathcal{E}(t)&=\frac{1}{2}\int_{\Omega_{\eta(t)}}
 |\bfu(t)|^2   \dx
 +
 \frac{1}{2}\int_{\omega}
|\partial_t\eta(t) |^2  \dy
+ \frac{1}{2}\int_{\omega}
|\partial_y^2\eta(t) |^2  \dy
\end{aligned}
\end{align} 
and hence
\begin{align*}
\sup_{I^\ast}\mathcal E(t)&+ \int_{I^\ast}\int_{\Omega_{\eta}}|\nabla \bfu |^2 \dx\dt+\int_{I^\ast}\int_{\omega}|\partial_t\partial_y\eta |^2 \dy\dt\\&\lesssim \mathcal E(0)+\int_{I^\ast}\int_{\Omega_\eta}|\bff|^2\dxt+\int_{I^\ast}\int_\omega|g|^2\dy\dt=:C_0
\end{align*}
for all $t\in I^\ast=(0,T^\ast)$.
This implies
\begin{align}
\sup_{I^\ast}\|\bfu\|_{L^2_x}^2+\int_{I^\ast}\|\nabla\bfu\|_{L^2_x}^2\dt\lesssim\,C_0,\label{eq:aprioriu}\\
\sup_{I^\ast}\|\partial_t\eta\|_{L^2_y}^2+\sup_{I^\ast}\|\nabla^2\eta\|_{L^2_y}^2+\int_{I^\ast}\|\partial_t\partial_y\eta\|_{L^2_y}^2\dt\lesssim\,C_0.\label{eq:apriorieta}
\end{align}

\subsection{The acceleraton estimate}\label{sec:accest}
The acceleration estimate is the heart of our analysis. It heavily relies on the elliptic estimate for the Stokes system in irregular domains given in Theorem \ref{thm:stokessteady}. A further difference to \cite{GraHil} is that we cannot work with a solenoidal extension operator (as explained in Section \ref{sec:accintro}).
Hence we must estimate the pressure function.
\begin{proposition}\label{prop2}
Suppose that the assumptions of Theorem \ref{thm:main} hold and let $(\eta,\bfu,\pi)$ be the unique local strong solution from Proposition \ref{prop:local}. Then we have the estimate
\begin{align}
\sup_{I^\ast}\int_{\Omega_\eta}&|\nabla\bfu|^2\dx+\int_{I^\ast}\int_{\Omega_\eta}\big(|\nabla^2\bfu|^2+|\partial_t\bfu|^2+|\nabla\pi|^2\big)\dxt\nonumber\\
 &+\sup_{I^\ast}\int_{\omega}\big(
|\partial_t\partial_y\eta |^2+
|\partial_y^3\eta |^2\big)\dy+\int_{I^\ast}\int_{\omega}
\big(|\partial_t\partial_y^2\eta |^2+|\partial_t^2\eta|^2\big)  \dy\dt\label{est:reg}\\
&\lesssim\int_{\Omega_{\eta_0}}|\nabla\bfu_0|^2\dx
+\int_\omega\big(|\partial_y^3\eta_0|^2+|\partial_y\eta_1|^2\big)\dy+\int_{I^\ast}\int_\omega|\partial_y g|^2\dy\dt+1,
\nonumber
\end{align}
where the hidden constant depends on $C_0$.
\end{proposition}
\begin{proof}
We aim at testing the structure equation with $\partial_t^2\eta$ and seek for an appropriate test-function for the momentum equation. Due to the condition $\bfu(t,x+\eta\bfn)=\partial_t\eta\bfn$ at the boundary we have
\begin{align*}
\frac{\D}{\D t}\bfu&=\partial_t\bfu(t,x+\eta\bfn)+\bfu(t,x+\eta\bfn)\cdot\nabla\bfu(t,x+\eta\bfn)\\
&=\partial_t\bfu(t,x+\eta\bfn)+\partial_t\eta\bfn\cdot\nabla\bfu(t,x+\eta\bfn)
\end{align*}
such that the material derivative of the velocity field is the corresponding test-function for the momentum equation.
We use
\begin{align*}
\bfphi&=\partial_t\bfu+\mathscr F_\eta(\partial_t\eta\bfn)\cdot\nabla\bfu
\end{align*}
and $\phi=\partial_t^2\eta$ as test-function. Here $\mathscr F_\eta$ is the extension operator introduced in Section \ref{sec:ext}. Note that we have
\begin{align}
\|\mathscr F_\eta(\partial_t\eta\bfn)\|_{L^2_x}&\lesssim \|\partial_t\eta\|_{L^2_y},\quad
\|\mathscr F_\eta(\partial_t\eta\bfn)\|_{W^{1,2}_x}\lesssim \|\partial_t\eta\|_{W^{1,2}_y},\label{eq:feta}
\end{align}
as a consequence of Lemma \ref{lem:3.8}.
 From the momentum equation in the strong form \eqref{1'} we obtain
\begin{align*}
\int_{I^\ast}\int_{\Omega_\eta}&\big(\partial_t\bfu+\bfu\cdot\nabla\bfu\big)\cdot\big(\partial_t\bfu+\mathscr F_\eta(\partial_t\eta\bfn)\cdot\nabla\bfu\big)\dxt\\
&=\int_{I^\ast}\int_{\Omega_\eta}\Div\bftau\cdot\big(\partial_t\bfu+\mathscr F_\eta(\partial_t\eta\bfn)\cdot\nabla\bfu\big)\dxt+\int_{I^\ast}\int_{\Omega_\eta}\bff\cdot\big(\partial_t\bfu+\mathscr F_\eta(\partial_t\eta\bfn)\cdot\nabla\bfu)\big)\dxt,
\end{align*}
where $\bftau=\nabla\bfu+\nabla\bfu^\top-\pi\mathbb I_{2\times 2}$ is the Cauchy stress.
On the other hand, from the structure equation \eqref{2'} multiplied by $\partial_t^2\eta$,
and the formal computation
\begin{align*}
\int_{0}^t\int_{\omega}\partial_y^2\eta\,\partial_t^2\partial_y^2\eta\dy\ds&=-\int_{0}^t\int_{\omega}\partial_t(\partial_y^3\eta\,\partial_t\partial_y\eta)\dy\ds+\int_{0}^t\int_\omega|\partial_t\partial_y^2\eta|^2\dy\ds\\
&\leq\tfrac{1}{2}\sup_{I^\ast}\int_\omega|\partial_t\partial_y\eta|^2\dy+2\sup_{I^\ast}\int_\omega|\partial_y^3\eta|^2\dy+\int_{I^\ast}\int_\omega|\partial_t\partial_y^2\eta|^2\dy\ds
\end{align*}
we obtain
\begin{align*}
\int_{I^\ast}&\int_\omega|\partial_t^2\eta|^2\dy\dt+\sup_{I^\ast}\int_\omega|\partial_t\partial_y\eta|^2\dy\dt\\&\lesssim\int_\omega|\partial_y\eta_1|^2\dy+\sup_{I^\ast}\int_\omega|\partial_y^3\eta|^2\dy+\int_{I^\ast}\int_\omega|\partial_t\partial_y^2\eta|^2\dy+\int_{I^\ast}\int_\omega(g+\bfF)\,\partial_t^2\eta\dy\dt
\end{align*}
with $\bfF=-\bfn \bftau\circ\bfvarphi_\eta\partial_y\bfvarphi^\perp_\eta$.
Combining both, using Reynold's transport theorem (applied to $\int_{\Omega_{\eta(t)}}|\nabla\bfu(t)|^2\dx$) and Young's inequality and writing
$$\mathscr F_\eta(\partial_t\eta\bfn)\cdot\nabla\bfu=\bfu\cdot\nabla\bfu+\mathscr F_\eta(\partial_t\eta\bfn)\cdot\nabla\bfu-\bfu\cdot\nabla\bfu$$
gives
\begin{align}\label{eq:0302}
\begin{aligned}
\sup_{I^\ast}&\int_{\Omega_\eta}|\nabla\bfu|^2\dx+\int_{I^\ast}\int_{\Omega_\eta}|\partial_t\bfu+\bfu\cdot\nabla\bfu|^2
+\int_{I^\ast}\int_\omega|\partial_t^2\eta|^2\dy\dt+\sup_{I^\ast}\int_\omega|\partial_t\partial_y\eta|^2\dy\dt\\
&\lesssim\int_{I^\ast}\int_{\Omega_\eta}|\bfu\cdot\nabla\bfu|^2\dxt+\int_{I^\ast}\int_{\partial\Omega_\eta}(\partial_t\eta\bfn)\circ\bfvarphi_\eta^{-1}\cdot\bfn_\eta\circ\bfvarphi_\eta^{-1}|\nabla\bfu|^2\dd\mathcal H^1\dt\\&-\int_{I^\ast}\int_{\Omega_\eta}\big(\partial_t\bfu+\bfu\cdot\nabla\bfu\big)\cdot\big(\mathscr F_\eta(\partial_t\eta\bfn)\cdot\nabla\bfu\big)\dx\\
&-\int_{I^\ast}\int_{\Omega_\eta}\nabla\bfu:\big(\mathscr F_\eta(\partial_t\eta\bfn)^\top\nabla^2\bfu+\nabla\mathscr F_\eta(\partial_t\eta\bfn)\nabla\bfu^\top\big)\dxt\\
&+\int_{I^\ast}\int_{\Omega_\eta}\pi\,\Div\big(\mathscr F_\eta(\partial_t\eta\bfn)\nabla\bfu\big)\dxt+\int_{I^\ast}\int_{\Omega_\eta}|\bff|^2\dxt+\int_{\Omega_{\eta_0}}|\nabla\bfu_0|^2\dx\\
&+\int_\omega|\partial_y\eta_1|^2\dy+\sup_{I^\ast}\int_\omega|\partial_y^3\eta|^2\dy+\int_{I^\ast}\int_\omega|\partial_t\partial_y^2\eta|^2\dy\dt+\int_{I^\ast}\int_\omega|g|^2\dy\dt\\
&=:\mathrm{I}+\dots+\mathrm{XI}.
\end{aligned}
\end{align}
In order to control the first term we make use of Theorem \ref{thm:stokessteady}. Due to \eqref{eq:apriorieta} its application can be justified by Remark \ref{rem:stokes}.
We estimate for $\kappa>0$ arbitrary by Ladyshenskaya's inequality (recalling that $\partial\Omega_\eta$ is Lipschitz uniformly in time by  \eqref{eq:apriorieta}) and \eqref{eq:aprioriu}
\begin{align}\label{eq:reg1}\begin{aligned}
\mathrm{I}&\leq \int_{I^\ast}\|\bfu\|^2_{L^4_x}\|\nabla\bfu\|_{L^4_x}^2\dt\leq\,c\int_{I^\ast}\|\bfu\|_{L^2_x}\|\bfu\|^2_{W^{1,2}_x}\|\nabla\bfu\|_{W^{1,2}_x}\dt\\
&\leq \,c\int_{I^\ast}\|\bfu\|^2_{W^{1,2}_x}\big(\|\partial_t\bfu+\bfu\cdot\nabla\bfu\|_{L^2_x}+\|\bff\|_{L^2_x}+\|\partial_t\eta\|_{W^{3/2,2}_y}\big)\dt\\
&\leq \,c(\kappa)\bigg(\int_{I^\ast}\|\nabla\bfu\|^4_{L^2_x}\dt+1\bigg)+\kappa\int_{I^\ast}\big(\|\partial_t\bfu+\bfu\cdot\nabla\bfu\|_{L^2_x}^2+\|\bff\|^2_{L^2_x}+\|\partial_t\eta\|^2_{W^{3/2,2}_y}\big)\dt,
\end{aligned}
\end{align}
where the first part of the $\kappa$-term can be absorbed in the left-hand side of \eqref{eq:0302}. Note that we also used the estimate
\begin{align}\label{again}
\|\partial_t\eta\bfn\circ \bfvarphi_\eta^{-1}\|_{W^{3/2,2}_y}\lesssim\|\partial_t\eta\|_{W^{3/2,2}_y},
\end{align}
which is a consequence of \eqref{eq:apriorieta} and the definition ${\bfvarphi}_\eta={\bfvarphi}+ \eta{\bfn}$. In fact,
$\bfvarphi_\eta^{-1}$ is uniformly bounded in time in the space
of Sobolev multipliers on $W^{3/2,2}(\omega)$ by \eqref{eq:MSa}, \eqref{eq:MSb} and \eqref{eq:SMPhiPsi} (together with the assumption $\partial_y\bfvarphi_\eta\neq0$)
such that the transformation rule \eqref{lem:9.4.1} applies.
Similarly to \eqref{eq:reg1} we obtain
\begin{align*}
\mathrm{III}&\leq\,\kappa\int_{I^\ast}\|\partial_t\bfu+\nabla\bfu\cdot\nabla\bfu\|_{L^2_x}^2\dt
+c(\kappa)\int_{I^\ast}\|\nabla\bfu\|_{L^4_x}^2\|\mathscr F_\eta(\partial_t\eta\bfn)\|^2_{L^4_x}\dt\\
&\leq\,\kappa\int_{I^\ast}\|\partial_t\bfu+\nabla\bfu\cdot\nabla\bfu\|_{L^2_x}^2\dt\\&+c(\kappa)\int_{I^\ast}\|\nabla\bfu\|_{L^2_x}\|\nabla\bfu\|_{W^{1,2}_x}\|\mathscr F_\eta(\partial_t\eta\bfn)\|_{L^2_x}\|\mathscr F_\eta(\partial_t\eta\bfn)\|_{W^{1,2}_x}\dt.
\end{align*}
Due to  \eqref{eq:apriorieta}, \eqref{eq:feta} and Theorem \ref{thm:stokessteady} (which applies by \eqref{eq:apriorieta}, cf. Remark \ref{rem:stokes}) the second term can be estimated by
\begin{align*}
\int_{I^\ast}\|\nabla\bfu\|_{L^2_x}&\|\nabla\bfu\|_{W^{1,2}_x}\|\mathscr F_\eta(\partial_t\eta\bfn)\|_{W^{1,2}_x}\dt\\&\leq\,c(\kappa')\bigg(\int_{I^\ast}\|\nabla\bfu\|_{L^2_x}^4\dt+\int_{I^\ast}\|\mathscr F_\eta(\partial_t\eta\bfn)\|_{W^{1,2}_x}^4\dt\bigg)+\kappa'\int_{I^\ast}\|\nabla\bfu\|^2_{W^{1,2}_x}\dt\\
&\leq \,c(\kappa')\bigg(\int_{I^\ast}\|\nabla\bfu\|^4_{L^2_x}\dt+\int_{I^\ast}\|\partial_t\eta\|_{W^{1,2}_y}^4\dt\bigg)\\
&+\kappa'\int_{I^\ast}\big(\|\partial_t\bfu+\bfu\cdot\nabla\bfu\|_{L^2_x}^2+\|\bff\|_{L^2_x}^2+\|\partial_t\eta\|^2_{W^{3/2,2}_y}\big)\dt,
\end{align*}
where $\kappa'>0$ is arbitrary (recall also \eqref{again}).
Furthermore, we have by the trace theorem (recall that the boundary of $\Omega_\eta$ is Lipschitz continuous uniformly in time by \eqref{eq:apriorieta}), Sobolev's embedding and interpolation\footnote{Analogously to \eqref{again} one can show $\|\partial_t\eta\bfn\circ \bfvarphi_\eta^{-1}\|_{W^{1,2}_y}\lesssim\|\partial_t\eta\|_{W^{1,2}_y}$.}
\begin{align*}
\mathrm{II}&\leq\int_{I^\ast}\|\partial_t\eta\circ \bfvarphi_\eta^{-1}\|_{L^\infty(\partial\Omega_\eta)}\|\bfu\|_{W^{1,2}(\partial\Omega_\eta)}^2\dy\leq\,c\int_{I^\ast}\|\partial_t\eta\|_{W^{1,2}(\omega)}\|\bfu\|_{W^{3/2,2}(\Omega_\eta)}^2\dy\\
&\leq\,c\int_{I^\ast}\|\partial_t\eta\|_{W^{1,2}_y}\|\bfu\|_{W^{1,2}_x}\|\bfu\|_{W^{2,2}_x}\dy\\
&\leq\,c\int_{I^\ast}\|\partial_t\eta\|_{W^{1,2}_y}\|\bfu\|_{W^{1,2}_x}\big(\|\partial_t\bfu+\bfu\cdot\nabla\bfu\|_{L^2_x}+\|\bff\|_{L^2_x}+\|\partial_t\eta\|_{W^{3/2,2}_y}\big)\dy\\
&\leq \,c(\kappa)\bigg(\int_{I^\ast}\|\nabla\bfu\|^4_{L^2_x}\dt+\int_{I^\ast}\|\partial_t\eta\|_{W^{1,2}_y}^4\dt\bigg)\\&+\kappa\int_{I^\ast}\big(\|\partial_t\bfu+\bfu\cdot\nabla\bfu\|_{L^2_x}^2+\|\bff\|_{L^2_x}^2+\|\partial_t\eta\|^2_{W^{3/2,2}_y}\big)\dt.
\end{align*}
Note that we used again Theorem \ref{thm:stokessteady} and \eqref{again}. Similarly, it holds
\begin{align*}
\mathrm{IV}\leq\,\kappa \int_{I^\ast}\big(\|\partial_t\bfu+\bfu\cdot\nabla\bfu\|_{L^2_x}^2+\|\bff\|_{L^2_x}^2+\|\partial_t\eta\|^2_{W^{3/2,2}_y}\big)\dt
+c(\kappa)\bigg(\int_{I^\ast}\|\nabla\bfu\|_{L^2_x}^{4}\dt+\int_{I^\ast}\|\partial_t\eta\|_{W^{1,2}_y}^{4}\dt+1\bigg).
\end{align*}
For $\mathrm{V}$ we write
\begin{align*}
\mathrm{V}=-\int_{I^\ast}\int_{\Omega_\eta}\nabla\pi\cdot\mathscr F_\eta(\partial_t\eta\bfn)\nabla\bfu\dxt+\int_{I^\ast}\int_{\partial\Omega_\eta}\pi\,\mathscr F_\eta(\partial_t\eta\bfn)\nabla\bfu\,\bfn_\eta\circ\bfvarphi_\eta^{-1}\,\dd\mathcal H^1\dt,
\end{align*}
where
\begin{align*}
\int_{I^\ast}&\int_{\Omega_\eta}\nabla\pi\cdot\mathscr F_\eta(\partial_t\eta\bfn)\nabla\bfu\dxt\\
&\leq\int_{I^\ast}\|\nabla\pi\|_{L^2_x}\,\|\mathscr F_\eta(\partial_t\eta\bfn)\|_{L^4_x}\|\nabla\bfu\|_{L^4_x}\dt\\
&\leq\,c \int_{I^\ast}\|\nabla\pi\|_{L^2_x}\,\|\mathscr F_\eta(\partial_t\eta\bfn)\|_{W^{1/2,2}_x}\|\nabla\bfu\|_{L^2_x}^{\frac{1}{2}}\|\nabla\bfu\|_{W^{1,2}_x}^{\frac{1}{2}}\dt\\
&\leq\,c \int_{I^\ast}\|\nabla\pi\|_{L^2_x}\|\mathscr F_\eta(\partial_t\eta\bfn)\|^{1/2}_{L^{2}_x}\|\mathscr F_\eta(\partial_t\eta\bfn)\|_{W^{1,2}_x}^{1/2}\|\nabla\bfu\|_{L^2_x}^{\frac{1}{2}}\|\nabla\bfu\|_{W^{1,2}_x}^{\frac{1}{2}}\dt\\
&\leq \kappa\int_{I^\ast}\|\nabla\pi\|_{L^2_x}^{2}\dt+c(\kappa)\int_{I^\ast}\|\nabla\bfu\|_{L^2_x}^{4}\dt+c(\kappa)\int_{I^\ast}\|\partial_t\eta\|_{W^{1,2}_y}^{4}\dt+\kappa\int_{I^\ast}\|\nabla\bfu\|_{W^{1,2}_x}^{2}\dt
\end{align*}
by \eqref{eq:feta} and \eqref{eq:apriorieta}. Note that we used again interpolation, the trace theorem as well as Sobolev's embedding $W^{1/2,2}(\Omega_\eta)\hookrightarrow L^4(\Omega_\eta)$ (recall again that the boundary $\partial\Omega_\eta$ is uniformly Lipschitz by \eqref{eq:apriorieta}). As far as the second term in $\mathrm{IV}$ is concerned we estimate 
\begin{align*}
\int_{I^\ast}\int_{\partial\Omega_\eta}\pi\,\mathscr F_\eta(\partial_t\eta\bfn)\nabla\bfu\,\bfn_\eta\circ\bfvarphi_\eta^{-1}\,\dd\mathcal H^1\dt&\leq\int_{I^\ast}\|\pi\|_{L^4(\partial\Omega_\eta)}\|\mathscr F_\eta(\partial_t\eta\bfn)\|_{L^2(\partial\Omega_\eta)}\|\nabla\bfu\|_{L^4(\partial\Omega_\eta)}\dt\\
&\leq\,c\int_{I^\ast}\|\pi\|_{W^{1/4,2}(\partial\Omega_\eta)}\|\nabla\bfu\|_{W^{1/4,2}(\partial\Omega_\eta)}\dt\\
&\leq\,c\int_{I^\ast}\|\pi\|_{W^{3/4,2}(\Omega_\eta)}\|\nabla\bfu\|_{W^{3/4,2}(\Omega_\eta)}\dt\\
&\leq\,c\int_{I^\ast}\|\pi\|_{W^{1,2}(\Omega_\eta)}\|\nabla\bfu\|_{W^{1,2}(\Omega_\eta)}^{3/4}\|\nabla\bfu\|_{L^{2}(\Omega_\eta)}^{1/4}\dt\\
&\leq\kappa\int_{I^\ast}\|\pi\|^2_{W^{1,2}_x}\dt+\kappa\int_{I^\ast}\|\nabla\bfu\|_{W^{1,2}_x}^{2}\dt+c(\kappa)\int_{I^\ast}\|\nabla\bfu\|_{L^{2}_x}^{2}\dt
\end{align*}
using Sobolev's embedding $W^{1/4,2}(\partial\Omega_\eta)\hookrightarrow L^4(\partial\Omega_\eta)$, \eqref{eq:feta} and \eqref{eq:apriorieta} and the trace embedding $W^{3/4,2}(\Omega_\eta)\hookrightarrow W^{1/4,2}(\partial\Omega_\eta)$. Different to first term in $\mathrm{V}$ above we must estimate here also the $L^2$-norm of the pressure, for which we use
\eqref{eq:pressure} (noticing that $\int_{\omega}\bfn\cdot\bfn_\eta|\partial_y\bfvarphi_\eta|\dy$ is stricly positive by our assumption of non-degeneracy). We have
\begin{align*}
\int_{I^\ast}\|\pi\|^2_{W^{1,2}_x}\dt&\lesssim \int_{I^\ast}\|\nabla\pi\|^2_{L^{2}_x}\dt+\int_{I^\ast}c_\pi^2\dt\\
&\lesssim \int_{I^\ast}\|\nabla\pi\|^2_{L^{2}_x}\dt+\int_{I^\ast}\int_\omega|\partial_t^2\eta|^2\dy\dt+\int_{I^\ast}\int_\omega|g|^2\dy\dt\\&+\int_{I^\ast}\|\pi_0\|_{L^{2}(\partial\Omega_\eta)}^2\dt+\int_{I^\ast}\|\nabla\bfu\|_{L^{2}(\partial\Omega_\eta)}^2\dt,
\end{align*}
where the last term can be estimated by
\begin{align*}
\int_{I^\ast}\|\nabla\bfu\|_{W^{1/2,2}(\Omega_\eta)}^2\dt
&\lesssim\int_{I^\ast}\|\nabla\bfu\|_{W^{1,2}(\Omega_\eta)}\|\nabla\bfu\|_{L^{2}(\Omega_\eta)}\dt\\
&\lesssim\int_{I^\ast}\|\nabla\bfu\|_{W^{1,2}_x}^{2}\dt+\int_{I^\ast}\|\nabla\bfu\|_{L^{2}_x}^{2}\dt.
\end{align*}
Similarly, we infer from Poincar\'e's inequality
\begin{align*}
\int_{I^\ast}\|\pi_0\|_{L^{2}(\partial\Omega_\eta)}^2\dt
\lesssim\int_{I^\ast}\|\nabla\pi_0\|_{L^{2}_x}^{2}\dt+\int_{I^\ast}\|\pi_0\|_{L^{2}_x}^{2}\dt\lesssim \int_{I^\ast}\|\nabla\pi_0\|_{L^{2}_x}^{2}\dt=\int_{I^\ast}\|\nabla\pi\|_{L^{2}_x}^{2}\dt
\end{align*}
using that $(\pi_0)_{\Omega_{\eta}}=0$ by definition.
 As for the estimates for
$\mathrm{I}-\mathrm{III}$ the $\kappa$ terms in the above can now be controlled by means of Theorem \ref{thm:stokessteady}.
Combining everything, choosing $\kappa$ small enough and using \eqref{eq:apriorieta} once more we conclude that
\begin{align}\label{eq:reg2}
\begin{aligned}
\sup_{I^\ast}\int_{\Omega_\eta}|\nabla\bfu|^2\dx&+\int_{I^\ast}\int_{\Omega_\eta}|\partial_t\bfu+\bfu\cdot\nabla\bfu|^2
+\int_{I^\ast}\int_\omega|\partial_t^2\eta|^2\dx\dt\\
&\lesssim \int_{I^\ast}\|\nabla\bfu\|^4_{L^2_x}\dt+\int_{I^\ast}\|\partial_t\partial_y\eta\|_{L^{2}_y}^4\dt+\int_{I^\ast}\|\bff\|_{L^2_x}^2\dt+\|\nabla\bfu_0\|_{L^2_x}^2\\&+\sup_{I^\ast}\|\partial_y^3\eta\|_{L^2_y}^2\dt+\int_{I^\ast}\|\partial_t\partial_y^2\eta\|_{L^2_y}^2\dt+\|\partial_y\eta_1\|^2_{L^2_y}+1.
\end{aligned}
\end{align}
Testing the structure equation by $\partial_t\partial_y^2\eta$ yields
\begin{align*}
\frac{1}{2}\sup_{I^\ast}\int_\omega|\partial_t\partial_y\eta|^2\dy\dt&+\int_{I^\ast}\int_\omega|\partial_t\partial_y^2\eta|^2\dy+\frac{1}{2}\sup_{I^\ast}\int_\omega|\partial_y^3\eta|^2\dy\\&=\frac{1}{2}\int_\omega|\partial_y\eta_1|^2\dy+\frac{1}{2}\int_\omega|\partial_y^3\eta_0|^2\dy+\int_{I^\ast}\int_\omega(g+\bfF)\,\partial_t\partial_y^2\eta\dy\dt,
\end{align*}
where (arguing as in \eqref{again} to control $\bfF$ by $\bftau$ and arguing as for $\mathrm{I}-\mathrm{III}$ and $\mathrm{V}$ to estimate $\bftau$)
\begin{align*}
\int_{I^\ast}&\int_\omega\bfF\cdot\partial_t\partial_y^2\eta\dy\dt\leq\int_{I^\ast}\|\bfF\|_{W_y^{1/2,2}(\omega)}\|\partial_t\partial_y^2\eta\|_{W^{-1/2,2}(\omega)}\dt\leq\,c\int_{I^\ast}\|\bftau\|_{W^{1/2,2}(\partial\Omega_\eta)}\|\partial_t\eta\|_{W^{3/2,2}(\omega)}\dt\\
&\leq\,c\int_{I^\ast}\big(\|\partial_t\bfu+\bfu\cdot\nabla\bfu\|_{L^{2}_x}+\|\bff\|_{L^2_x}+\|\partial_t\eta\|_{W^{3/2,2}_y}+\|\partial_t^2\eta\|_{L^2_y}+\|g\|_{L^2_y}\big)\|\partial_t\eta\|_{W^{3/2,2}_y}\dt\\
&\leq\kappa\int_{I^\ast}\big(\|\partial_t\bfu+\bfu\cdot\nabla\bfu\|_{L^{2}_x}^2+\|\bff\|_{L^2_x}^2+\|\partial_t^2\eta\|_{L^2_y}^2+\|g\|_{L^2_y}^2\big)\dt+c(\kappa)\int_{I^\ast}\|\partial_t\eta\|_{W^{3/2,2}_y}^2\dt\\
&\leq\kappa\int_{I^\ast}\big(\|\partial_t\bfu+\bfu\cdot\nabla\bfu\|_{L^{2}_x}^2+\|\partial_t^2\eta\|_{L^2_y}^2+\|\partial_t\eta\|_{W^{2,2}_y}^2+\|\bff\|_{L^2_x}^2\big)\dt\\&+c(\kappa)\int_{I^\ast}\|\partial_t\eta\|_{W^{1,2}_y}^2\dt+c(\kappa)\int_{I^\ast}\|g\|_{L^{2}_y}^2\dt
\end{align*}
using interpolation in the last step.
Hence we have
\begin{align}\label{eq:reg3}
\begin{aligned}
\sup_{I^\ast}&\int_\omega|\partial_t\partial_y\eta|^2\dx\dt+\int_{I^\ast}\int_\omega|\partial_t\partial_y^2\eta|^2\dy+\sup_{I^\ast}\int_\omega|\partial_y^3\eta|^2\dy\\&\leq\kappa\int_{I^\ast}\big(\|\partial_t\bfu+\bfu\cdot\nabla\bfu\|_{L^{2}}^2+\|\partial_t\partial_y^2\eta\|_{L^{2}_y}^2+\|\partial_t^2\eta\|_{L^2_y}^2\big)\dt+ c(\kappa)\tilde C_0
\end{aligned}
\end{align}
by \eqref{eq:apriorieta}. Here we denoted
\begin{align*}
\tilde C_0=C_0+\int_\omega|\partial_y\eta_1|^2\dy+\int_\omega|\partial_y^3\eta_0|^2\dy.
\end{align*}
Combining \eqref{eq:reg2} and \eqref{eq:reg3} implies
\begin{align*}
\sup_{I^\ast}\int_{\Omega_\eta}&|\nabla\bfu|^2\dx+\int_{I^\ast}\int_{\Omega_\eta}|\partial_t\bfu+\bfu\cdot\nabla\bfu|^2
+\int_{I^\ast}\int_\omega|\partial_t^2\eta|^2\dy\dt\\
+\sup_{I^\ast}&\int_\omega|\partial_t\partial_y\eta|^2\dx\dt+\int_{I^\ast}\int_\omega|\partial_t\partial_y^2\eta|^2\dx+\sup_{I^\ast}\int_\omega|\partial_y^3\eta|^2\dy\\
&\leq\,c\bigg(\int_{I^\ast}\big(\|\nabla\bfu\|^2_{L^2_x}+\|\partial_t\partial_y\eta\|_{L^2_y}^2\big)^2\dt+1\bigg).
\end{align*}
By Gronwall's lemma, using that $\int_{I^\ast}\big(\|\nabla\bfu\|^2_{L^2_x}+\|\partial_t\partial_y\eta\|_{L^2_y}^2\big)\dt\leq\,c$ by \eqref{eq:aprioriu} and \eqref{eq:apriorieta} we obtain
\begin{align}\label{eq:reg4}
\begin{aligned}
\sup_{I^\ast}\int_{\Omega_\eta}&|\nabla\bfu|^2\dx+\int_{I^\ast}\int_{\Omega_\eta}|\partial_t\bfu+\bfu\cdot\nabla\bfu|^2\dx
+\int_{I^\ast}\int_\omega|\partial_t^2\eta|^2\dy\dt\leq\,c,\\
\sup_{I^\ast}\int_\omega&|\partial_t\partial_y\eta|^2\dy\dt+\int_{I^\ast}\int_\omega|\partial_t\partial_y^2\eta|^2\dy+\sup_{I^\ast}\int_\omega|\partial_y^3\eta|^2\dy\leq c.
\end{aligned}
\end{align}
We can now use the momentum equation and Theorem \ref{thm:stokessteady} again to obtain (recall \eqref{again})
\begin{align}\label{eq:reg5}
\begin{aligned}
\int_{I^\ast}&\int_{\Omega_\eta}|\nabla^2\bfu|^2\dx\dt+\int_{I^\ast}\int_{\Omega_\eta}|\nabla\pi|^2\dx\dt\\&\leq\,c\int_{I^\ast}\int_{\Omega_\eta}|\partial_t\bfu+\bfu\cdot\nabla\bfu|^2\dxt+\int_{I^*}\|\bff\|_{L^2_x}^2\dt
+\int_{I^\ast}\|\partial_t\eta\|_{W^{3/2,2}_y}^2\dt\leq\,c,
\end{aligned}
\end{align}
which completes the proof.
\end{proof}

\begin{proof}[Proof of Theorem \ref{thm:main}]
By Propositions \ref{prop2} and \ref{prop:local} we can obtain a strong solution in the interval $(T^\ast,2T^\ast)$ with initial data $\bfu(T^\ast),\eta(T^\ast),\partial_t\eta(T^\ast)$ with a corresponding regularity estimate.
This procedure can now be repated until the moving boundary approaches a self-intersection or degenerates (that is $\partial_y\bfvarphi_\eta(T,y)=0$ for some $y\in\omega$). In the latter case Theorem \ref{thm:stokessteady} is not applicable anymore, cf. Remark \ref{rem:stokes}.
\end{proof}


\section*{Compliance with Ethical Standards}\label{conflicts}
\smallskip
\par\noindent
{\bf Conflict of Interest}. The author declares that he has no conflict of interest.

\smallskip
\par\noindent
{\bf Data Availability}. Data sharing is not applicable to this article as no datasets were generated or analysed during the current study.

\end{document}